\documentclass[12pt]{amsart} 

\usepackage{amsthm,amssymb}
\usepackage{fullpage}
\usepackage[utf8]{inputenc}
\usepackage{graphicx} 
\usepackage{subcaption}
\usepackage{amssymb}
\usepackage{xcolor}
\usepackage{float}
\usepackage{soul,xcolor}

\usepackage{soul,xcolor}

\newcommand{\stw}{\mathbb{S}^2}

\newtheorem{proposition}{Proposition}
\newtheorem{theorem}{Theorem}
\newtheorem{lemma}{Lemma}
\newtheorem{corollary}{Corollary}
\newtheorem{question}{Question}

\newtheorem{remark}{Remark}

\begin{document}
\setstcolor{red}
\title[Self-dual Maps I : antipodality]{Self-dual Maps I : antipodality} 

\thanks{$^1$ Partially supported by CONACyT 166306 and PAPIIT-UNAM IN112614}
\thanks{$^2$ Partially supported by grant PICS07848 and INSMI-CNRS}
\author[Luis Montejano]{Luis Montejano$^1$}
\address{Instituto de Matem\'aticas, Universidad Nacional A. de M\'exico at Quer\'etaro
Quer\'etaro, M\'exico, CP. 07360}
\email{luis@im.unam.mx}
\author[Jorge L. Ram\'irez Alfons\'in]{Jorge L. Ram\'irez Alfons\'in$^2$}
\address{
UMI2924 - J.-C. Yoccoz, CNRS-IMPA, Brazil and Univ.\ Montpellier, France }
\email{jorge.ramirez-alfonsin@umontpellier.fr}
\author[Ivan Rasskin]{Ivan Rasskin}
\address{IMAG, Univ.\ Montpellier, CNRS, Montpellier, France}
\email{ivan.rasskin@umontpellier.fr}



\begin{abstract} A self-dual map $G$ is said to be {\em antipodally self-dual} if the dual map $G^*$ is antipodal embedded in $\stw$ with respect to $G$. In this paper, we investigate necessary and/or sufficient conditions for a map to be antipodally self-dual. In particular, we present a combinatorial characterization for map $G$ to be antipodally self-dual in terms of certain {\em involutive labelings}.
The latter lead us to obtain necessary conditions for a map to be {\em strongly involutive} (a notion relevant for its connection with convex geometric problems). We also investigate the relation of antipodally self-dual maps and the notion of {\em antipodally symmetric} maps. It turns out that the latter is a very helpful tool to study questions concerning the {\em symmetry} as well as the {\em amphicheirality} of {\em links}.
\end{abstract}

\maketitle

\section{Introduction}

Let $G$ be a {\em map}, that is, a graph cellularly embedded in the sphere. Then $G=(V,E,F)$ has a natural {\em geometric dual} $G^*=(V^*,E^*,F^*)$ where each face in $F$ correspond to a vertex in $V^*$ and two vertices in $V^*$ are adjacent if the corresponding faces in $G$ share an edge. 
A map $G$ is called {\em self-dual} if there is a bijection from $V$ and $F$ to $V^*$ and $F^*$ which reverses inclusion.
\smallskip
 
Self-dual maps have been the subject of numerous investigations in different fronts : self-dual polyhedra and ranks \cite{GS}, isometries in $\stw$ \cite{SS}, eigenvalues of $h$-graphs \cite{T}, rigidity \cite{SC}, tilings \cite{SS1}, etc.
\smallskip

A self-dual map $G$ is said to be {\em antipodally self-dual} if the dual map $G^*$ is antipodally embedded with respect to $G$. In other words, the map $G$ is antipodally self-dual if the following holds for any $x\in \mathbb{S}^2$
\smallskip

1) if  $x\in V(G)$ then $-x\in V(G^*)$ and
\smallskip

2) if $x\in e \in E(G)$ then $-x\in e^*\in E(G^*)$, that is, $e^*$ is antipodally embedded in $\mathbb{S}^2$ with respect to the embedding of $e$.
\smallskip

Antipodally self-dual maps are closely related with the notion of {\em strongly involutive} maps (see beginning of Section \ref{subsec;strong}) and thus relevant for their connection with convex geometric problems as the well-known V\'azsonyi's problem on {\em ball polyhedra} (as reported in \cite{E}, see also \cite{KS}), the chromatic number of {\em distance graphs} on the sphere \cite{L} and {\em Reuleaux polyhedra} \cite{MMO}. As we will see, antipodally self-dual maps are also closely related with the notion of {\em antipodally symmetric} maps. The latter turns out very useful to study questions concerning the {\em symmetry} as well as the {\em amphicheirality} of links, see \cite{MRAR2}.
\smallskip

The main goal of this paper is to investigate necessary and/or sufficient conditions for a map to be antipodally self-dual.  
\smallskip

The paper is organized as follows. In the next section, we give a brief overview of some notions on self-dual maps needed for the rest of the paper. Given a map $G$, we recall three special close related maps ({\em medial graph $med(G)$, square graph} $G^\square$ and {\em vertex-face incidence graph} $I(G)$) that turn out to be very useful for our propose.  
\smallskip

In Section \ref{sec;antipodally self-dual}, we first recall some classical results between isometries in $\stw$ and maps. We then present a result  giving necessary conditions of an antipodally self-dual map $G$ in terms of {\em symmetric cycles} in $I(G)$ (Theorem \ref{theom;ant1}). Afterwards, we discuss the connection between antipodally self-dual maps and {\em strongly involutive} maps and give a combinatorial characterization for a map $G$ to be antipodally self-dual in terms of certain {\em involutive labelings} of $I(G)^\square$ (Theorem \ref{thm;fixedpoint}). As a consequence, we obtain necessary conditions for a map $G$ to be strongly involutive in terms of $I(G)^\square$ (Corollary \ref{cor;stronginv}).
\smallskip

In Section \ref{sec:ex}, we characterize three different infinite families of antipodally self-dual maps (Propositions \ref{prop;wheel}, \ref{prop;ear} and \ref{prop;pancake}). We also present a more general construction (Theorem \ref{thm;adh}). 
\smallskip

In Section \ref{sec;self-anti}, we study antipodally symmetric maps. Besides many properties, we show that if $G$ is an antipodally self-dual map then both $med(G)$ and $I(G)$ are antipodally symmetric maps (Lemma \ref{lem:key}).

\section{Maps preliminaries}\label{sec;selfdual}

Let $G$ be a planar graph. A {\em map} of $G=(V,E,F)$ is the image of an embedding of $G$ into $\mathbb{S}^2$ where the set of vertices are a collection of distinct points in $\mathbb{S}^2$ and the set of edges are a collection of Jordan curves joining two points in $V$ satisfying that $\alpha\cap\alpha'$ is either empty or a point in the endpoints for any pair of Jordan curves $\alpha$ and $\alpha'$.  Any embedding of the topological realization of $G$ into $\mathbb{S}^2$ partitions the 2-sphere into simply connected regions of $\mathbb{S}^2\setminus G$ called the {\em faces} $F$ of the embedding.
\smallskip

Given a map $G$, we may construct the {\em dual map} $G^*=(V^*,E^*,F^*)$ by placing a vertex
$f^*$ in the interior of each face $f$ of $G$, and for each edge $e$ of $M$ draw a {\em dual edge} $e^*$ connecting the vertices $f^*_1$ and $f^*_2$ (corresponding to the two faces $f_1$ and $f_2$ sharing edge $e$) by crossing $e$ transversely. We denote by $X(G,G^*)$ the set of intersection points of map $G$ and map $G^*$.
\smallskip

Two maps $G_1=(V_1,E_1,F_1)$ and $G_2=(V_2,E_2,F_2)$ of the same graph are
{\em isomorphic} if there is an isomorphism $\phi:(V_1,E_1,F_1)\rightarrow (V_2,E_2,F_2)$ preserving incidences. We say that a map $G=(V,E,F)$ is a {\em self-dual map} if the maps $G=(V,E,F)$ and $G^*=(V^*,E^*,F^*)$ are isomorphic, that is, there is an isomorphism  $\phi:(V,E,F)\rightarrow (F^*,E^*,V^*)$ preserving incidences. 
\smallskip

Given maps $G=(V,E,F)$ and $G^*=(V^*,E^*,F^*)$ we define the following auxiliaries maps.
\smallskip

- The {\em squares graph of $G$} is the map $G^\square$ obtained by the simultaneous drawing of $G\cup G^*$ with all the edges split at the intersection points of an edge $e$ with its dual edge $e^*$. We thus have that every face of $G^\square$ is a square formed by half-edges of $G$ and $G^*$. 
\smallskip

For each square face in $G^\square$, we define two types of diagonals: the {\em intersecting diagonal} which is the edge joining the intersections points and the {\em incidence diagonal} which is the edge joining a vertex in $V(G)$ to a vertex in $V(G^*)$. 
\smallskip

- The {\em vertex-face incidence graph} is the map $I(G)$ having as vertices $V\cup V^*$ and as edges are all the incidence diagonals of $G^\square$.
\smallskip

- The {\em medial of $G$} is the map $med(G)$ having as vertices the set of intersections points of $E\cap E^*$ and as edges the set of all the intersecting diagonals of $G^\square$.

\begin{figure}[H]
\centering
\includegraphics[width=1\linewidth]{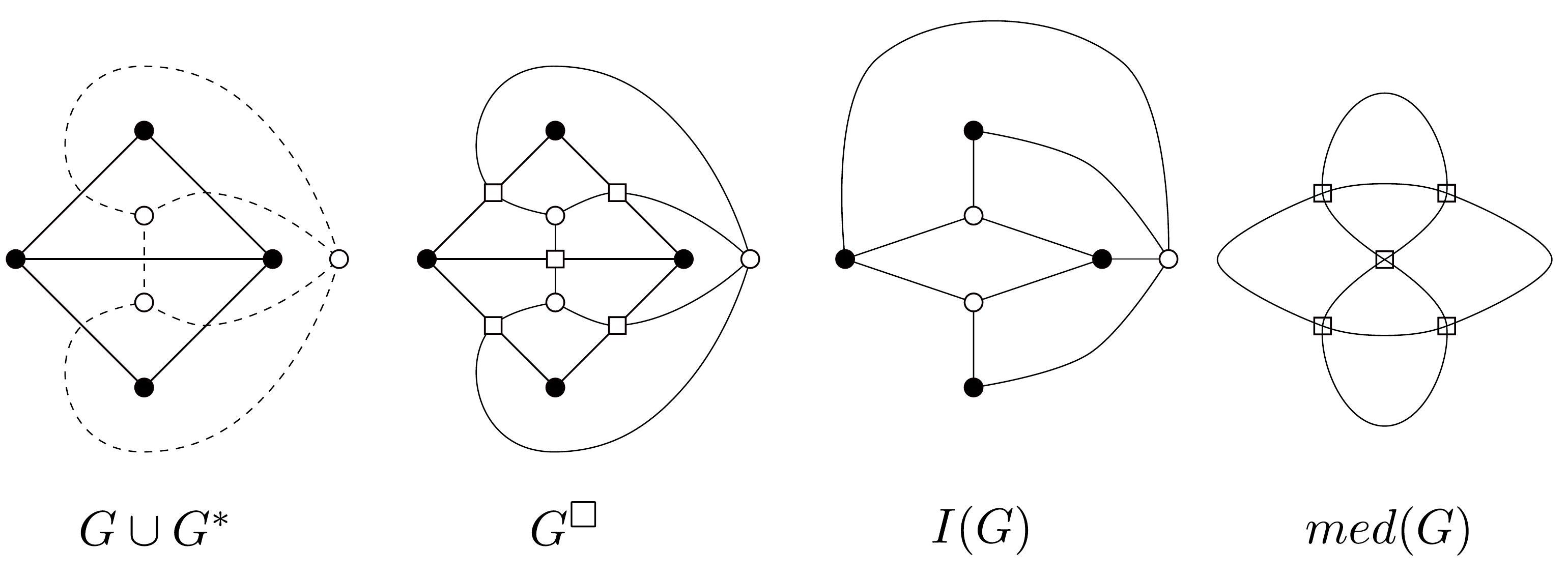}
\caption{A map and its dual, the squares graph, the vertex-face incidence graph and the medial.}
\label{fig11a}
\end{figure}

Throughout the paper, we will represent the vertices of $G$ with black circles, the vertices of $G^*$ with white circles and the intersection points with white squares and the vertices of the medial with transparent squares. 

Notice that $I(G)$ and $med(G)$ are dual from each other for any map $G$. Hence, we can construct the squares graph of the vertex-face incidence graph which it turns out to be very useful for our propose. 
\begin{figure}[H]
\centering
\includegraphics[width=0.8\linewidth]{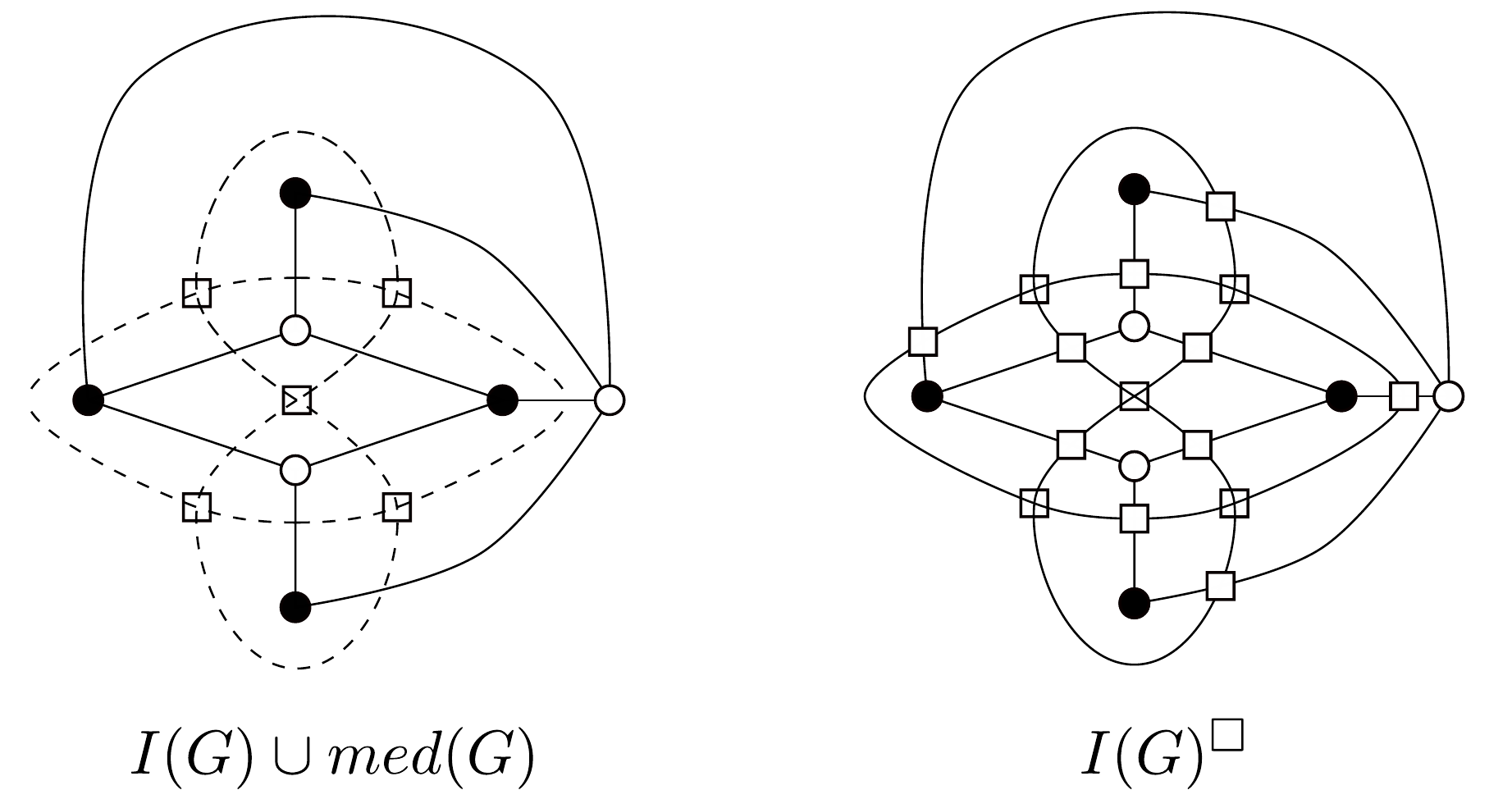}
\caption{(Right) $I(G)$ (straight edges and black and vertices in white circles) and $med(G)$ (dashed edges and vertices in transparent squares) (Left) Graph $I(G)^\Box$.}
\label{fig11}
\end{figure}

\section{antipodally self-dual maps}\label{sec;antipodally self-dual}

We recall that $Aut(G)$ is the group formed by the set of all {\em automorphism} of $G$ (i.e., the set  of isomorphisms of $G$ into itself).  We will denote by $Iso(G)$ the set of all {\em duality isomorphisms} of $G$ into $G^*$. We notice that $Iso(G)$ is not a group since the composition of any two of them is an automorphism. 
\smallskip

Let us suppose that $G=(V,E,F)$ is a self-dual map so that there is a bijection 
 $\phi:(V,E,F)\rightarrow (F^*,E^*,V^*)$. Following $\phi$ with the correspondence $*$ gives a permutation on $V\cup E\cup F$ which preserve incidences but reverses dimension of the elements. The collection of all such permutations or {\em self-dualities} generate a group $Dual(G)=Aut(G)\cup Iso(G)$ in which the automorphisms $Aut(G)$ are contained as a subgroup of index 2.   
\smallskip

It is known \cite[Lemma 1]{SS} that for a given map $G$ there is an homeomorphism $\rho$ of $\mathbb{S}^2$ to itself such that for every $\sigma\in Aut(G)$ we have that $\rho\sigma\in Isom(\mathbb{S}^2)$ where $Isom(\stw)$ is the group of isometries of the 2-sphere. In other words, any planar graph $G$ can be drawn on the 2-sphere such that any automorphism of $G$ act as an isometry of the sphere.  This was extended in \cite{SS} by showing that given any self-dual graph $G$ there are maps $G$ and $G^*$ so that $Dual(G)$ is realized as a group of spherical isometries. 

From now on, we will denote by $\widehat{G}=\rho(G)$ and $\widehat{\sigma}=\rho \sigma$ for a certain homeomorphism $\rho$ satisfying the above property.
\smallskip

A self-dual map $G$ is {\em antipodally self-dual} if $-\widehat{G}=\widehat{G}^*$ where $-G$ is the map consisting of the set of points $\{-x\in \stw \ |\  x\in G\}$. 
\smallskip

Let us present a result giving necessary combinatorial conditions for a map to be antipodally self-dual. 
By a {\em symmetric cycle} $C$ of a planar graph $G$ we mean  there is an automorphism $\sigma(G)$ such that $\sigma(C)=C$ and $\sigma(int(C))=ext(C)$, that is, the induced graph in the interior of $C$ is isomorphic to the induced graph in the exterior of $C$. 

\begin{theorem}\label{theom;ant1} Let $G$ be antipodally self-dual. Then, $I(G)$ always admit at least one symmetric cycle. Moreover, all symmetric cycles in $I(G)$ are of length $2n$ with $n\ge 1$ odd.
\end{theorem}

We will prove Theorem \ref{theom;ant1} at the end of Section \ref{sec;self-anti} where the notion of {\em antipodally symmetric} is discussed (and needed for the proof). 

\begin{remark}\label{rem;dd-ant} An antipodally self-dual map $G$ induces an involutive  self-dual isomorphism $\sigma: V(G)\longrightarrow V^*(G)$. The converse is not necessarily true, there are self-dual graphs not admitting an antipodally self-dual map. For instance, the graph $G'$ illustrated in Figure \ref{fig11z} is self-dual but it is not antipodally self-dual. Indeed, it can be easily checked that $I(G')$ admits a symmetric cycle of length 8 (implying that $G'$ is not antipodally self-dual, by Theorem \ref{theom;ant1}), see Figure \ref{fig11z}. 
\end{remark}

\begin{figure}[H]
\centering
\includegraphics[width=0.8\linewidth]{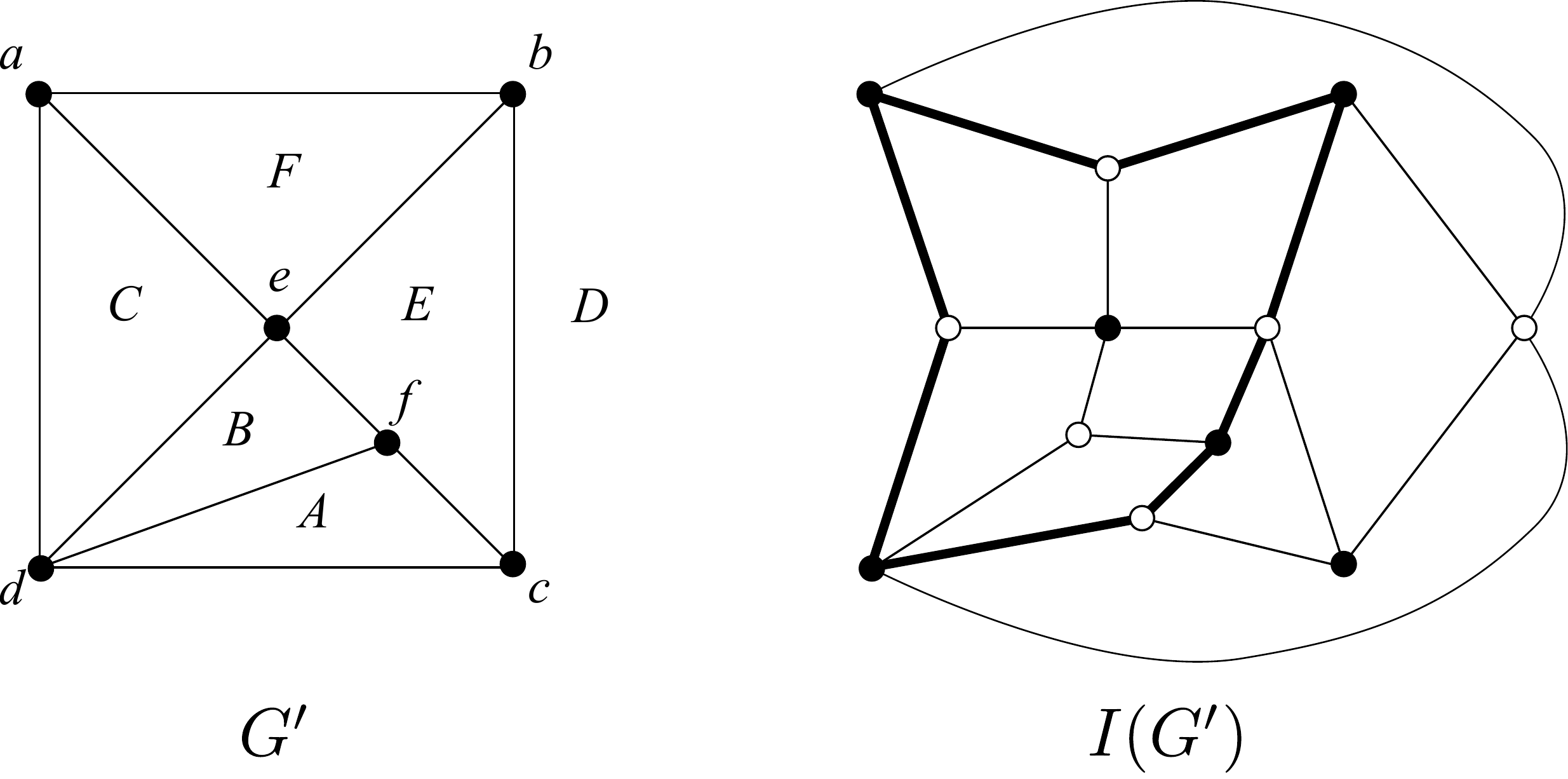}
\caption{(Left) a self-dual map $G'$ (the isomorphism $\sigma$ is given by $\sigma(a)=A, \sigma(b)=B, \sigma(c)=C,$ etc.) not admitting a strongly involutive dual-isomorphism (Right) $I(G')$ admitting a symmetric cycle of length 8 (bold edges).}
\label{fig11z}
\end{figure}

\subsection{Strongly involutive maps and involutive labelings}\label{subsec;strong}
 Let $G$ be a self-dual graph with duality isomorphism $\sigma: G\longrightarrow G^*$. We say that $G$ is {\em strongly involutive} if the following conditions are satisfied:
\smallskip

a) for each pair of vertices $u,v\in V(G), u\in \sigma(v)$ if and only if $v\in\sigma(u)$ and
\smallskip

b) for every vertex $v\in V(G)$, we have that $v\not\in\sigma(v)$.
\smallskip

We notice that a) is equivalent to say that $\sigma^2=id$.
\smallskip

The above conditions are the combinatorial counterpart (in the 3-dimensional case) of a more general geometric object called {\em strong self-dual polytopes}, first introduced by Lov\'asz in \cite{L}. Antipodally self-dual maps are closely related with strongly involutive isomorphism. Indeed, in \cite[Theorem 9]{BMPRA}, it was proved that if $G$ is strongly involutive then $G$ is antipodally self-dual. As we will see below, the latter is a straight forward consequence of Theorem \ref{thm;fixedpoint} (see Corollary \ref{cor;stronginv}).
\smallskip

Let $G=(V,E,F)$ be a map and let $X^+=\{x_1,\ldots,x_m\}$ and $X^-=\{\overline{x}_1,\ldots,\overline{x}_m\}$ be two sets with $1\le m\le |V|$ and the property $\overline{\overline{x_i}}=x_i$. Let $\mathcal{P}(X^+\cup X^-)$ be the set of subsets of $X^+\cup X^-$. An \textit{involutive labeling} of $G$ is a function $\Lambda:V\rightarrow\mathcal{P}(X^+\cup X^-)$ satisfying the following properties: 
\smallskip

\begin{itemize}
\item[$(i)$]$|\Lambda(v)|=1,2$ for every $v\in V$.
\item[$(ii)$]If $|\Lambda(v)|=2$ then $\Lambda(v)=\{x_i,\overline{x}_i\}$ for some $1\le i \le m$. In this case, we say that $v$ is a \textit{fixed vertex} of $\Lambda$ and we write $x_i=\overline{x}_i$ (instead of $\{x_i,\overline{x}_i\}$).
\item[$(iii)$]$\Lambda(u)\cap\Lambda(v)\not=\emptyset$ if and only if $u=v$.
\item[$(iv)$]$\{\Lambda^{-1}(x_i),\Lambda^{-1}(x_j)\}\in E$ if and only if $\{\Lambda^{-1}(\overline{x}_i),\Lambda^{-1}(\overline{x}_j)\}\in E$ where\\
 $\Lambda^{-1}(x_i):=\{v\in V\mid x_i\in\Lambda(v)\}$.
\end{itemize}
\smallskip






Let $G^\square=(V^\square,E^\square,F^\square)$ be the square graph associated to a map $G=(V,E,F)$. Recall that $V^\square=V_V\cup V_E \cup V_F$ where $V_V$ are the vertices of $G$, $V_E$ are the vertices on the edges of $G$ and $V_F$ are the vertices of $G^*$ (one for each face of $G$).
\smallskip

\begin{remark}\label{rem;inv-aut} An involutive labeling of $I(G)^\square$ naturally induces an automorphism of $I(G)$ 

$$\begin{array}{rll}
\sigma_\Lambda : & V\cup V^* \to & V\cup V^*\\
& v \mapsto & u\\
\end{array}$$

where $\Lambda(u)=\overline{\Lambda(v)}$ (the adjacency preserving property of $\sigma_\Lambda$ is obtained from $(ii)$).
\smallskip

a) If vertex $v$ was assigned labels $k$ and $\bar k$ (and thus $k=\bar k$) then it will be a fixed vertex under $\sigma_\Lambda$. 
\smallskip

b) $\sigma_\Lambda^2=Id$.
\smallskip

c) $\sigma_\Lambda$ corresponds to an involutive  duality isomorphism $\sigma:G\rightarrow G^*$ if and only if 
the labels of the black vertices are the opposite to those of the white vertices in $I^\square(G)$.  
\smallskip

\end{remark}

\begin{remark} Let $G$ be a self-dual map. We have that $G$ is strongly involutive if and only if  $I(G)$ admits an involutive labeling without edges which extremes are labeled by $k$ and $\bar k$.
\end{remark} 

\subsection{Characterizing antipodally self-dual maps}\label{subsec;isomdual-ant} We are interested in giving necessary and sufficient combinatorial conditions for a map to be antipodally self-dual.
\smallskip

\begin{remark}\label{rem;sigmaSqr} We have that any $\sigma\in Aut(G)$ naturally induces $\sigma^\square\in Aut(G^\square)$ with $\sigma^\square$ preserving incidences, that is, if $v_V\in V_V$ is adjacent to $v_E\in V_E$ (resp. $v_E\in V_E$ is adjacent to $v_F\in V_F$) then $\sigma^\square(v_V)$ is adjacent to $\sigma^\square(v_E)$ (resp. $\sigma^\square(v_E)$ is adjacent to $\sigma^\square(v_F)$) and where $V_V, V_E$ and $V_F$ are mapped to $V_V, V_E$ and $V_F$ respectively. We finally notice that there might exist $\gamma\in Aut(G^\square)$ not necessarily arising from an automorphism of $G$.
\end{remark}
  
\begin{lemma}\label{lem:pointsfixes} Let $H$ be a map and let $\sigma\in Aut(H)$. Then, $\widehat \sigma$ has a fixed point in $\stw$ if and only if $\sigma^\square$ has a fixed vertex in $H^\square$.
\end{lemma}

 \begin{proof} Let $x\in\stw$. A point $x$ corresponds to a vertex on $H^\square$, say $x^\square$, which lies  properly on either $V, E$ or $F$. If $\widehat \sigma(x)=x$ then $\sigma^\square(x^\square)=x^\square$.
 \smallskip
 
Conversely, let $v\in V^\square=\{V_V\cup V_E\cup V_F\}$ such that $\sigma^\square(v)=v$. We have three cases.
\smallskip
 
Case 1) $v\in V_V$. Then, the point $v\in\stw$ is such that $\widehat \sigma(v)=v$.  
\smallskip
 
Case 2) $v\in V_E$. Suppose $v$ lies properly on an edge $e$. We know that the isometry $\widehat \sigma$ maps $e$ into itself. Since $e$ is topologically equivalent to $\mathbb{B}^1$ then $\widehat \sigma$ is a continuous function sending $\mathbb{B}^1$ to itself. Therefore, by the Brouwer fixed-point theorem there is $x\in e$ such that $\widehat \sigma(x)=x$. 
\smallskip
 
Case 3) $v\in V_F$. Suppose $v$ lies properly on a face $f$. We proceed as in the Case 2.
The isometry $\widehat \sigma$ maps $f$ into itself. Since $f$ is topologically equivalent to $\mathbb{B}^2$ then $\widehat \sigma$ is a continuous function sending $\mathbb{B}^2$ to itself. Therefore, by the Brouwer fixed-point theorem there is $x\in f$ such that $\widehat \sigma(x)=x$. 
 \end{proof}
 
\begin{theorem}\label{thm;fixedpoint} Let $G=(V,E,F)$ be a self-dual map. Then, $G$ is antipodally self-dual if and only if $I(G)^\square$ admits an involutive labeling without fixed vertices. 
 \end{theorem}

\begin{proof} Suppose that $G$ is antipodally self-dual. Therefore,  there is $\widehat{G}$ isomorphic to $G$ such that $-\widehat{G}=\widehat{G}^*$. Let $a:x\mapsto-x$ be the antipodal mapping of $\stw$.  We have that $a$ naturally induces the automorphisms $a_I\in Aut(I(\widehat{G}))$ and $a^\square\in Aut(I(\widehat{G})^\square)$ . Furthermore, since $a$ is the antipodal mapping then
\begin{itemize}
\item $a_I^2=Id$ (implying that $I(G)^\square$ admits an involutive labeling on its vertices) and
\item $a_I$ has no fixed points of $\stw$. Therefore, by Lemma \ref{lem:pointsfixes}, $a^\square$ has no fixed vertices and thus the above involutive labeling of $I(G)^\square$ has no fixed vertices.
\end{itemize}
We finally notice that an involutive labeling of $I(\widehat{G})^\square$ is also an involutive labeling of $I(G)^\square$.
\smallskip

Conversely, suppose that $I(G)^\square$ admits an involutive labeling without fixed vertices.
By Lemma \ref{lem:pointsfixes}, $\widehat \sigma(I(G))$ has not a fixed point in $\stw$. Now, there are three sphere isometries such that $\sigma^2=Id$ : rotation of $\pi$ degree, reflexion on a hyperplane and the antipodal function. Among them, it is the antipodal function the only without fixed points. Moreover, since  $\sigma : G\rightarrow G^*$ then $\sigma$ sends vertices of $G$ to vertices of $G^*$. Therefore, $G$ is antipodally self-dual.
 \end{proof}
For the involutive labelings of squares graphs,, we shall use integers (and their opposites) for vertices of type $V_V$, letters (and their opposites) for vertices of type $V_F$ and greek letters (and their opposites) for vertices of type $V_E$.
On one hand Figure \ref{fig:40} illustrates a self-dual map $G$ and $I(G)^\square$ together with an involutive labeling without fixed vertices. Therefore, as a consequence of Theorem \ref{thm;fixedpoint}, $G$ is antipodally self-dual. On the other hand, Figure \ref{fig32a} illustrates an involutive labeling of the 4-wheel  $W_4$ with $I(W_4)^\square$ admitting two fixed vertices. In fact, it can be checked that any involutive labeling of $I(W_4)^\square$ admits at least one fixed vertex since $W_4$ is not antipodally self-dual (see Proposition \ref{prop;wheel}).  
\begin{figure}[H]
    \centering
    \includegraphics[width=0.89\textwidth]{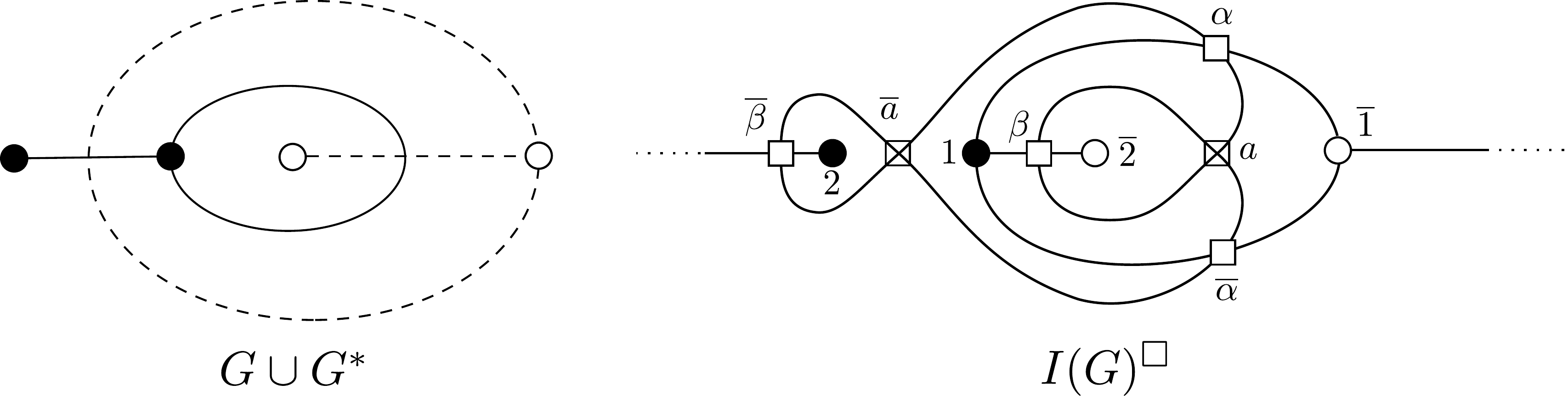}
    \caption{(Left) A self-dual map $G$ (straight edges and black vertices) and $G^*$ (dashed edges and white vertices). It can easily be checked that $G$ do not admit a strongly involutive isomorphism. (Right) An involutive labeling of $I(G)^\square$ without fixed vertices.}
    \label{fig:40}
\end{figure}
\begin{figure}[H]
\centering
\includegraphics[width=.9\linewidth]{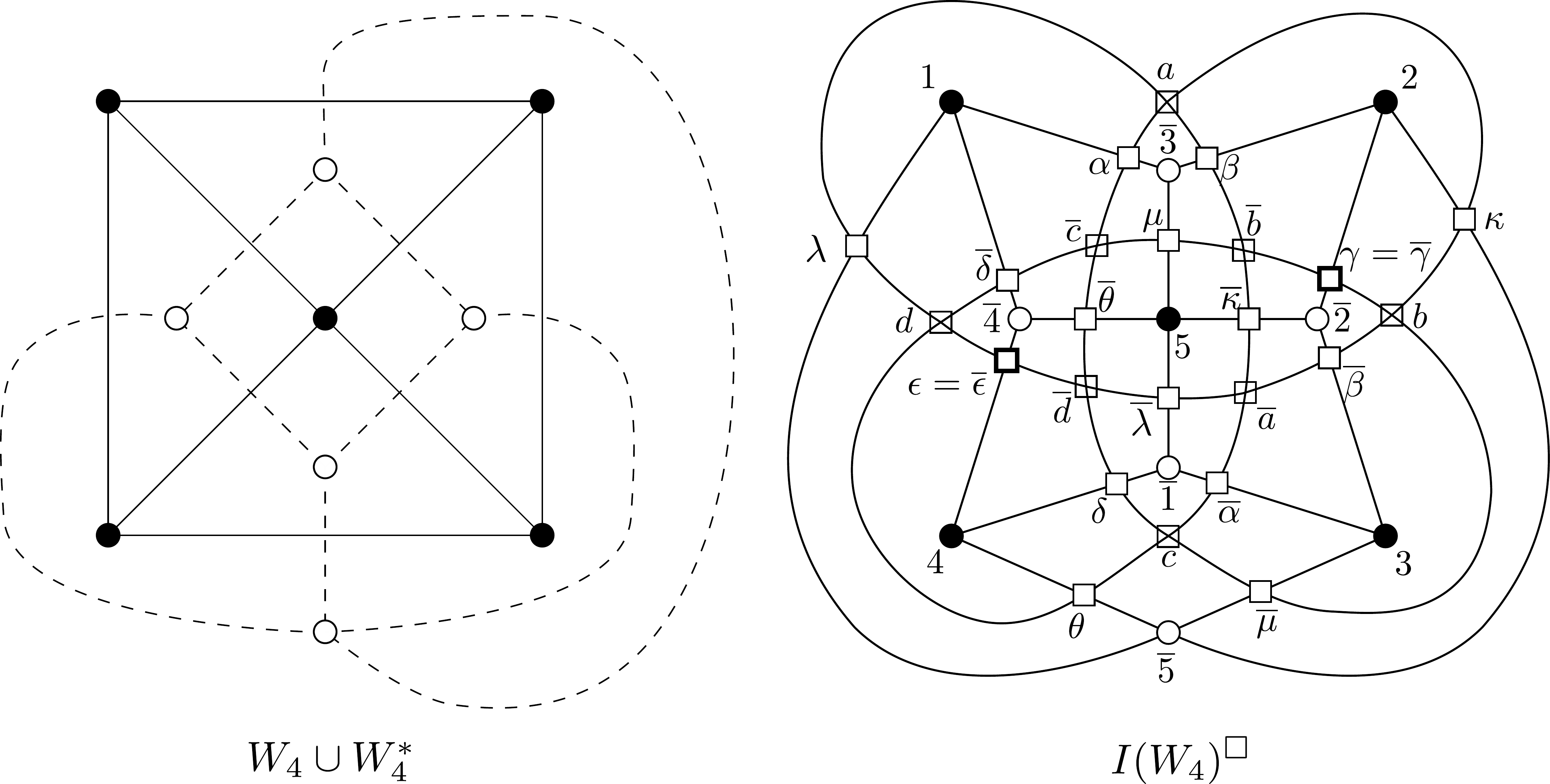}
\caption{(Left) The self-dual map $W_4$ with its dual.
(Right) $I^\square(W_4)$ together with an involutive labelling with two fixed vertices de type $V_E$: $\gamma=\bar \gamma$ and $\epsilon=\bar\epsilon$ (bold squares).}
\label{fig32a}
\end{figure}

\begin{corollary}\label{cor;nodual-ant} Let $G$ be a self-dual map. If there is a black vertex of $I(G)$ connected to each white vertex of $I(G)$ by an odd number of edges then $G$ is not antipodally self-dual.
\end{corollary}

\begin{proof}Let $v$ be such a black vertex. Since $v$ is connected to all the white vertices 
then for any involutive labeling $\Lambda$ of $I(G)^\square$ there is an edge in $I(G)$ with ends labeled with $k$ and $\overline{k}$. By Remark \ref{rem;inv-aut}, the automorphism $\sigma_\Lambda(G)$ maps an edge with ends labeled $\{k,\overline{k}\}$ to an edge with ends labeled  $\{k,\overline{k}\}$. Since, by hypothesis, there is an odd number of edges then there must be an edge mapped to itself which correspond to a fixed vertex in $V(I(G)^\square)$.  Therefore, by Theorem \ref{thm;fixedpoint}, $G$ is not antipodally self-dual.
\end{proof}

Figure \ref{fig32b} illustrates a graph in which $I(G)$ has a vertex 
in $V(G)$ adjacent to  each vertex of $G^*$ by an odd number of edges (and thus, by Corollary \ref{cor;nodual-ant}, $G$ is not antipodally self-dual).

\begin{figure}[H]
\centering
\includegraphics[width=.6\linewidth]{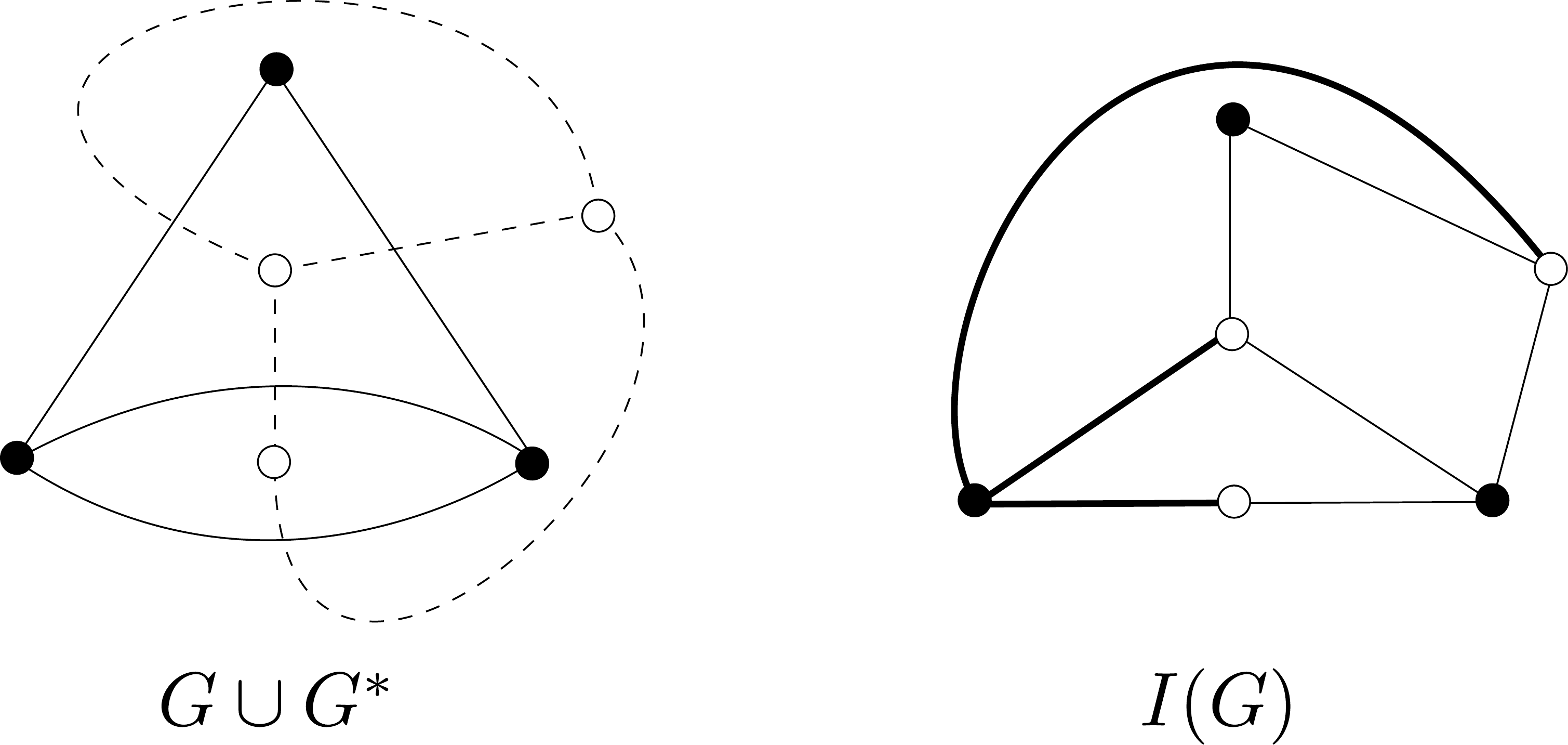}
\caption{(Left) A self-dual map $G$ with its dual. (Right) $I(G)$ with a black vertex joined to each white vertex by an odd number of edges (in bold).}
\label{fig32b}
\end{figure}

\begin{corollary}\label{cor;stronginv} Let $G$ be a self-dual map. If $G$ is strongly involutive then $G$ is antipodally self-dual.
\end{corollary}

\begin{proof}  We shall show that $I(G)^\square$ admits an involutive labeling without fixed vertices. The result then follows by Theorem \ref{thm;fixedpoint}. 

Let $\sigma : G\longrightarrow G^*$ be a duality isomorphism. We thus have that $\sigma$ does not fix vertices. Recall that if $G$ is strongly involutive then $\sigma$ verifies 
\smallskip

a) for each pair of vertices $u,v\in V(G), u\in \sigma(v)$ if and only if $v\in\sigma(u)$ and
\smallskip

b) for every vertex $v\in V(G)$, we have that $v\not\in\sigma(v)$.

As remarked above, a) is equivalent to say that $\sigma^2=id$.
We clearly have that $\sigma$ does not fix vertices since it maps vertices of $G$ to vertices of $G^*$. 
The latter implies that $\sigma_I$ does not fix vertices in $I(G)$ and thus neither $\sigma^\square$ in $I(G)^\square$.
\smallskip

Now, by combining conditions (a) and (b) we obtain that $u\not\in\sigma (u)$ for every vertex $u$ in $G^*$. The latter implies that $I(G)$ does not admit an edge with extremes labeled with
 $k$ and $\bar k$ and so $\sigma^\square$ does not fix vertices of type $V_E$ (i.e., arising from edges of $I(G)$) in $I(G)^\square$.
\smallskip

We finally claim that $\sigma^\square$ does not fix vertices of type $V_F$ (i.e., arising from faces of $I(G)$) in $I(G)^\square$. We proceed by contradiction, suppose that  $\sigma\square$ fixes a vertex $u_f$ arising from a face $f$ of $I(G)$. Let $f$ be the face in $I(G)$ corresponding to $\sigma(u_f)$. 
Recall that all the faces in $I(G)$ are squares, suppose that $f=\{w,x,y,z\}$ with $w, y\in V(G)$ and $x, z\in V(G^*)$ and $f'=\{w',x',y',z'\}$ with $w', y'\in V(G)$ and $x', z'\in V(G^*)$.
\smallskip

Since $\sigma^\square$ fixes $u_f$ then $\sigma(u_f)=u_f$ but this happen only if $\{\sigma(w),\sigma(y)\}=\{w',y')\}$ and $\{\sigma(x),\sigma(z)\}=\{x',z')\}$. The latter implies the existence of an edge with extremes labeled $k$ and $\bar k$, which is not possible.
\end{proof}

We notice that the converse of Corollary \ref{cor;stronginv} is not necessarily true. Indeed, there might be a non strongly involutive map $G$ with $I(G)^\square$ admitting an involutive labeling without fixed vertices (and thus $G$ antipodally self-dual, by Theorem \ref{thm;fixedpoint}), see Figure \ref{fig:40}.

\section{Infinite families}\label{sec:ex}

We give below some infinite families having antipodally self-dual maps. For, it is given an appropriate strongly involutive duality-isomorphism.  We will present a result giving sufficient and necessary conditions for a map to be  antipodally self-dual  in Section \ref{subsec;isomdual-ant} (Theorem \ref{thm;fixedpoint}) which can also be used to verify that the below families are antipodally self-dual. The latter is based on {\em involutive isometries} in $\mathbb{S}^2$ without fixed points.  

\subsection{The wheel} Let $n\ge 3$ be an integer. The $n$-{\em wheel}, denoted by $W_n$, is the graph consisting of an $n$-cycle with a center joined to each vertex of the cycle. 

\begin{proposition}\label{prop;wheel} The $n$-wheel is antipodally self-dual if and only if $n\ge 3$ is odd.
\end{proposition}

\begin{proof} It can be easily checked that $W_n$ admits a strongly involutive duality-isomorphism for any odd integer $n\ge 3$, see Figure \ref{fig18aa}. Moreover, if $n$ is even then $I(W_n)$ admits a symmetric cycle of length $2k$ with $k$ even, see Figure \ref{fig18a2}. Thus, by Theorem \ref{theom;ant1}, $W_n$ is not antipodally self-dual. 
\end{proof}

\begin{figure}[H]
\centering
\includegraphics[width=0.45\linewidth]{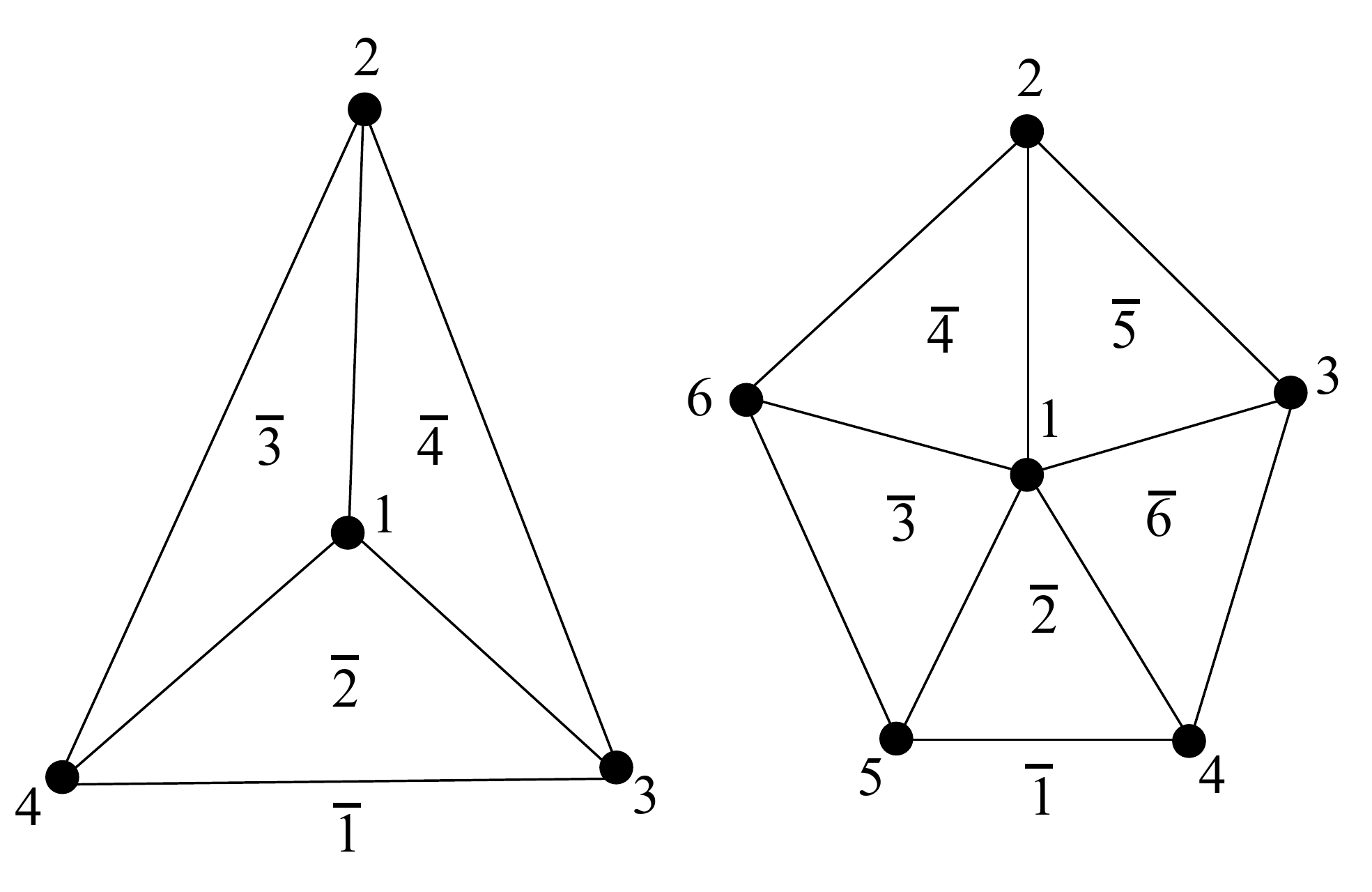}
\caption{3-wheel and 5-wheel together with a strongly involutive duality-isomorphism given by $\sigma(k)=\bar k$.}
\label{fig18aa}
\end{figure}

\begin{figure}[H]
\centering
\includegraphics[width=0.26\linewidth]{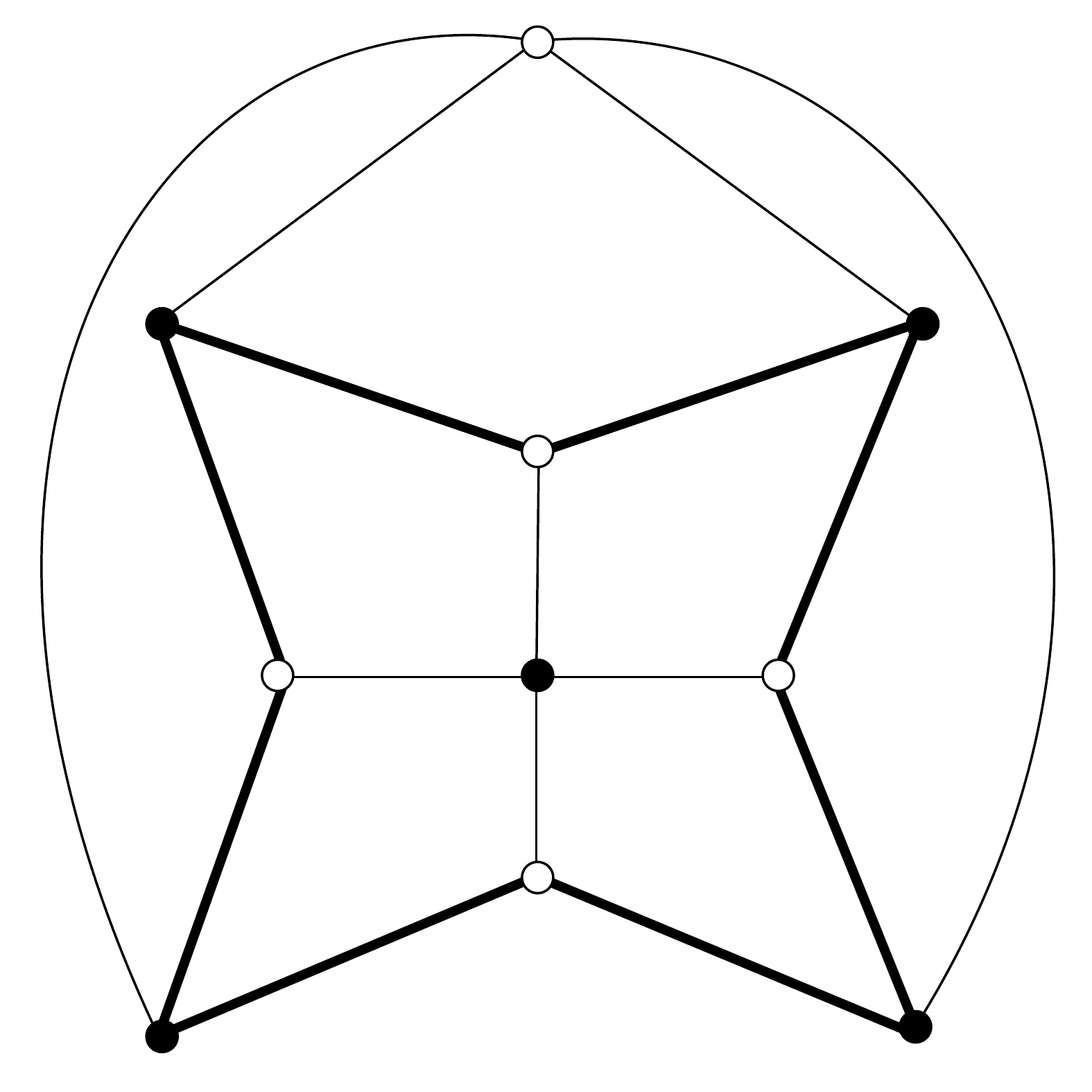}
\caption{$I(W_4)$ admitting a symmetric cycle (bold edges) of length 8.}
\label{fig18a2}
\end{figure}

Figure \ref{fig13} (a) shows that $W_3$ admits a antipodally self-dual map. One can easily mimic this embedding for any  odd integer $n\ge 3$. Figure \ref{fig18} illustrates the case $n=5$.

\begin{figure}[H]
\centering
\includegraphics[width=0.4\linewidth]{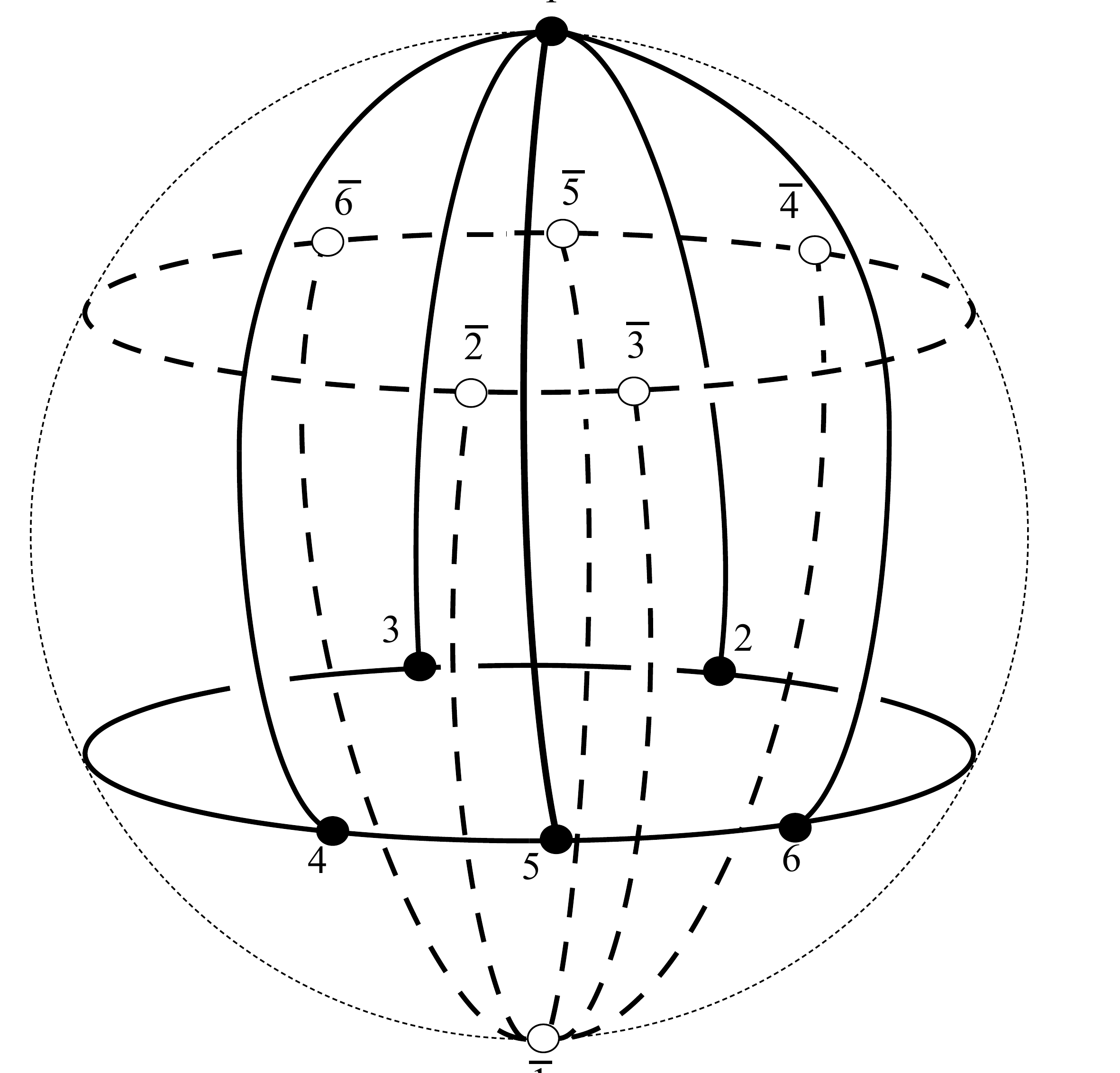}
\caption{A antipodally self-dual map of $W_5$ (straight edges and black vertices) and its dual (dashed edges and white vertices). Antipodal vertices are given by $k$ and $\bar k$}
\label{fig18}
\end{figure}

\subsection{The $n$-ear} Let $n\ge 3$ be an integer. The $n$-{\em ear}, denoted by $E_n$ is the graph consisting of a $n$-cycle with an {\em ear} added on each edge and a center is joined to each ear, see Figure  \ref{fig18a}
\smallskip

\begin{proposition}\label{prop;ear} The $n$-ear is antipodally self-dual if and only if $n\ge 4$ is even.
\end{proposition}

\begin{proof} It can be easily checked that $E_n$ admits a strongly involutive duality-isomorphism for any even integer $n\ge 4$, see Figure \ref{fig18a}. Moreover, if $n$ is odd then $I(E_n)$ admits a symmetric cycle of length $2k$ with $k$ even, see Figure \ref{fig18a21}. Thus, by Theorem \ref{theom;ant1}, $W_n$ is not antipodally self-dual. 
\end{proof}

\begin{figure}[H]
\centering
\includegraphics[width=0.7\linewidth]{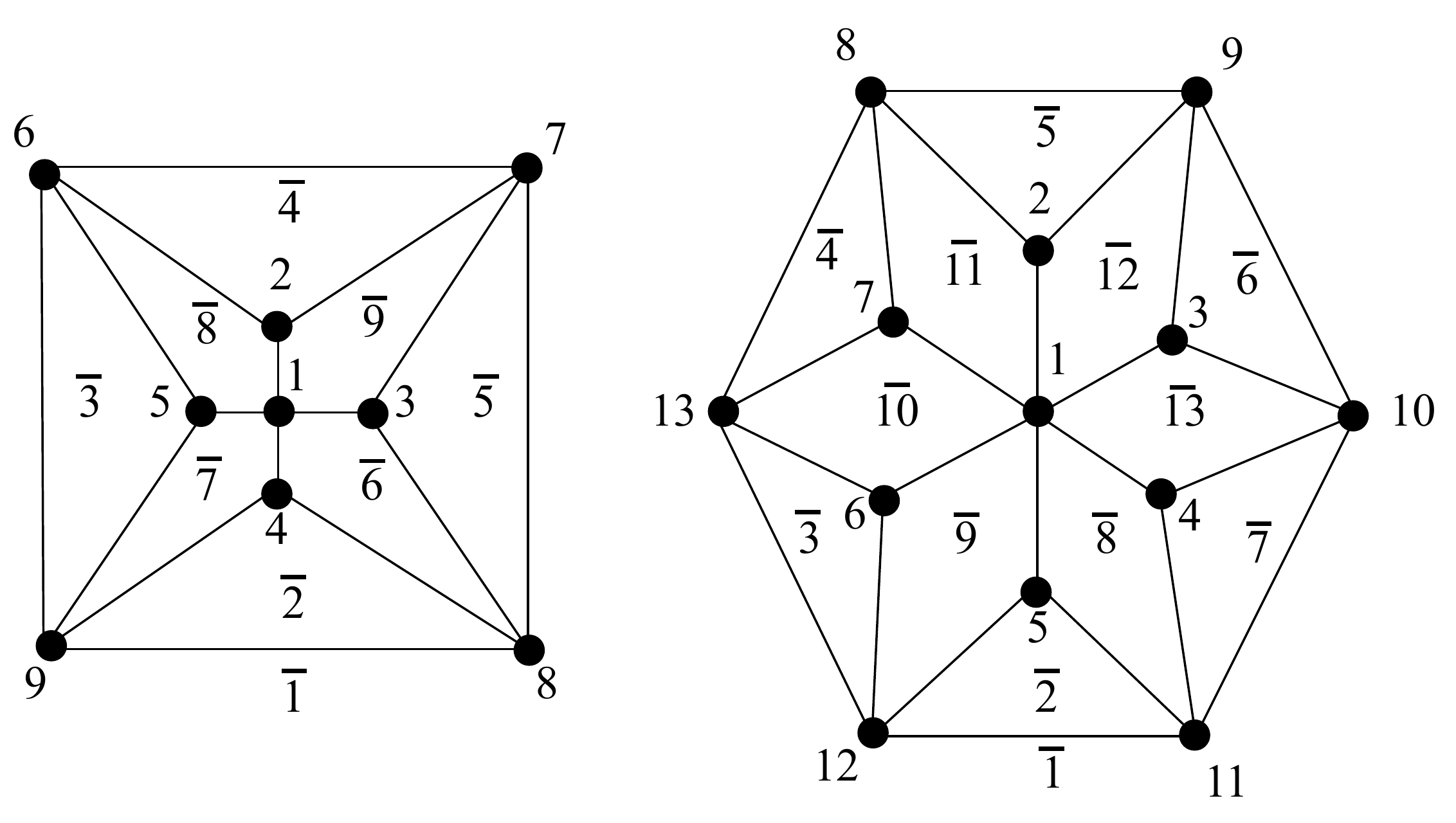}
\caption{The 4-ear and 6-ear graphs together with strongly involutive duality-isomorphisms given by $\sigma(k)=\bar k$.}
\label{fig18a}
\end{figure}

\begin{figure}[H]
\centering
\includegraphics[width=0.32\linewidth]{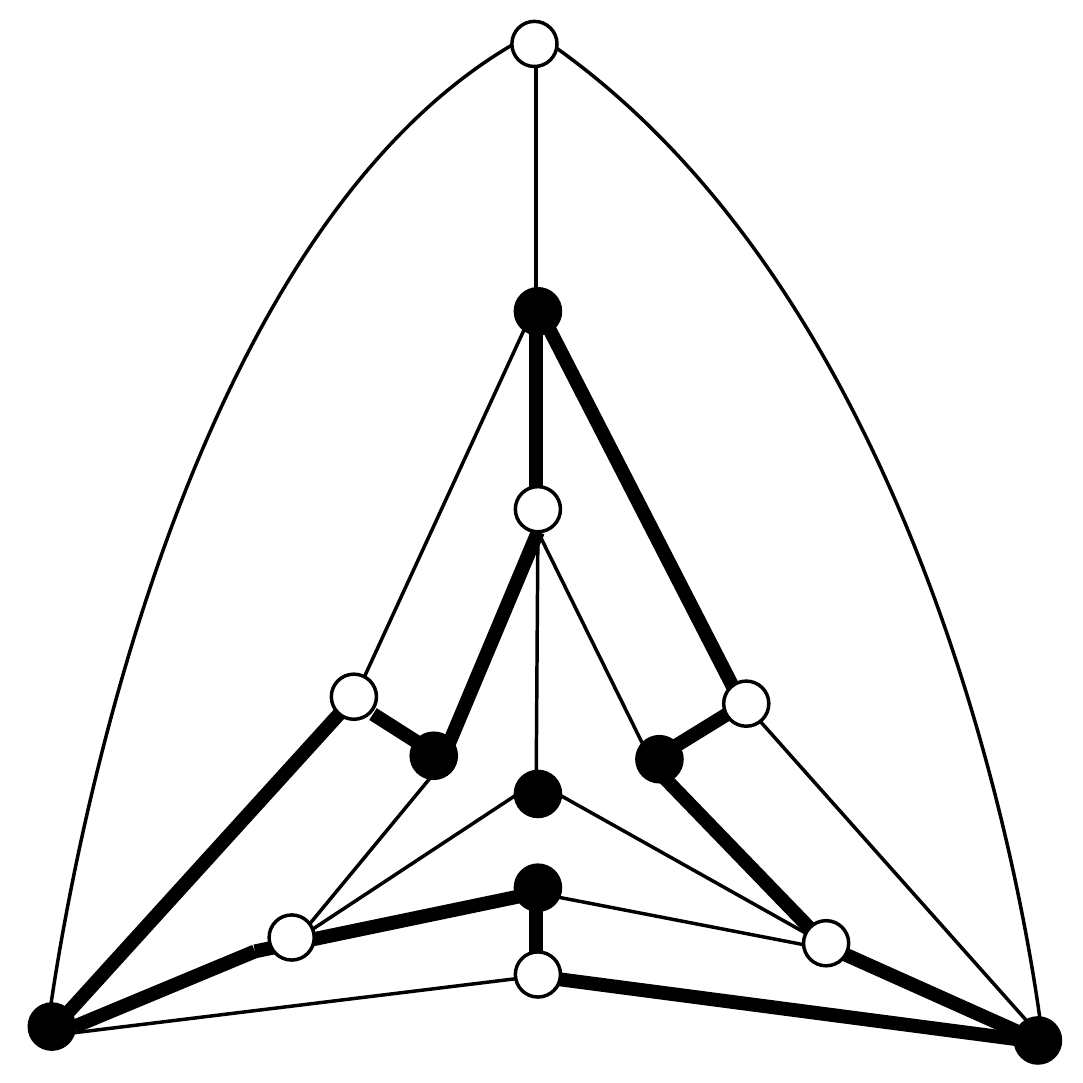}
\caption{$I(3$-ear) admitting a symmetric cycle (bold edges) of length 12.}
\label{fig18a21}
\end{figure}

The map $E_4$ given in Figure \ref{fig18b} shows that $4$-ear graph is antipodally self-dual. One can easily mimic this embedding for any even integer $n\ge 4$.

\begin{figure}[H]
\centering
\includegraphics[width=0.84\linewidth]{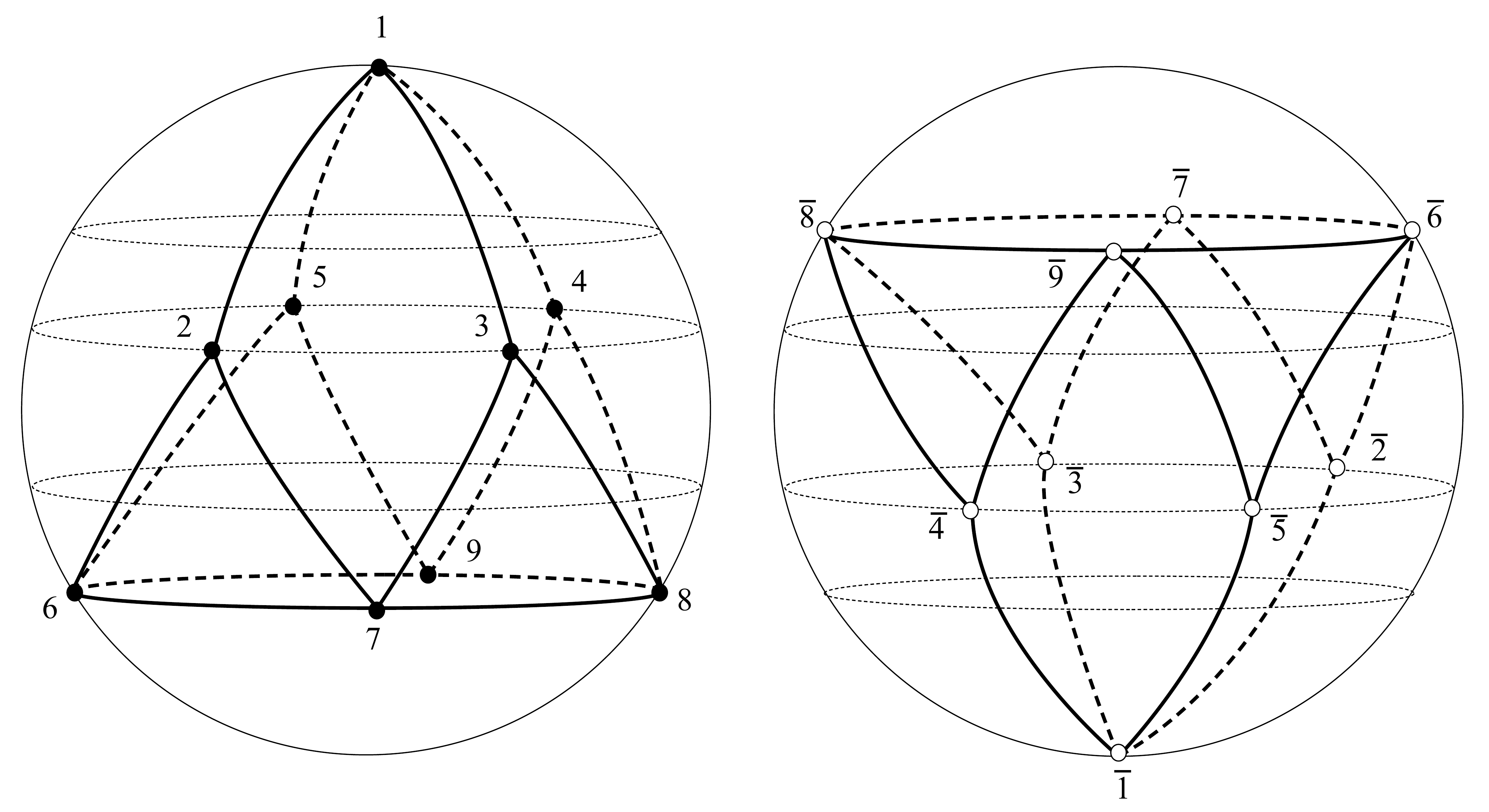}
\caption{A antipodally self-dual map of $E_4$ (black vertices) and its dual (white vertices). Antipodal vertices are given by $k$ and $\bar k$.}
\label{fig18b}
\end{figure}

\subsection{The $(n,\ell)$-pancake} Let $n\ge 3$ and $l\ge 1$ be integers. The $(n,\ell)$-{\em pancake}, denoted by $P_n^\ell$, is the graph consisting of $\ell$ cycles $\{v_1^1,\dots ,v_n^1\},\dots , \{v_1^\ell,\dots ,v_n^\ell\}$, a vertex $v_i^0$ and edges $\{v_i^{j-1},v_i^{j}\}$ for each $j=1,\dots ,n$ and all $i$, see Figure  \ref{fig18c}.
\smallskip

\begin{proposition}\label{prop;pancake} The $(n,\ell)$-pancake is antipodally self-dual if and only if $n\ge 3$ is odd for all $\ell\ge 1$.
\end{proposition}

\begin{proof} It can be easily checked that $(n,\ell)$-pancake admits a strongly involutive duality-isomorphism for all integers $n\ge 3, \ell\ge 1$ with $n$ odd, see Figure \ref{fig18c}. Moreover, if $n$ is even then $I((n,\ell)$pancake$)$ admits a symmetric cycle of length $2k$ with $k$ even, see Figure \ref{fig18c1}. Thus, by Theorem \ref{theom;ant1}, $(n,\ell)$-pancake is not antipodally self-dual. 
\end{proof}

\begin{figure}[H]
\centering
\includegraphics[width=0.75\linewidth]{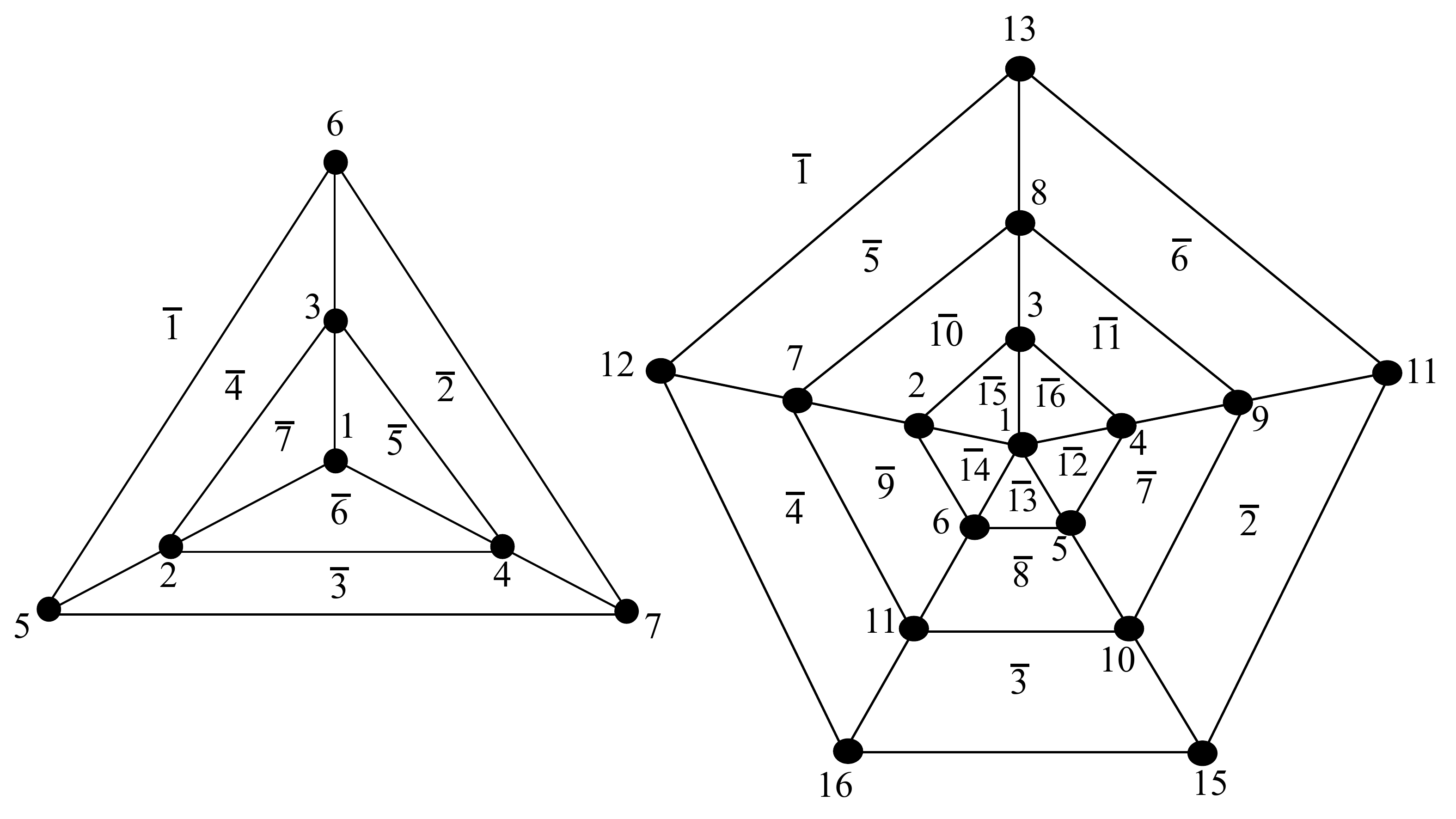}
\caption{$P_3^2$ and $P_5^3$ together with a strongly involutive duality-isomorphism given by $\sigma(k)=\bar k$.}
\label{fig18c}
\end{figure}

\begin{figure}[H]
\centering
\includegraphics[width=0.37\linewidth]{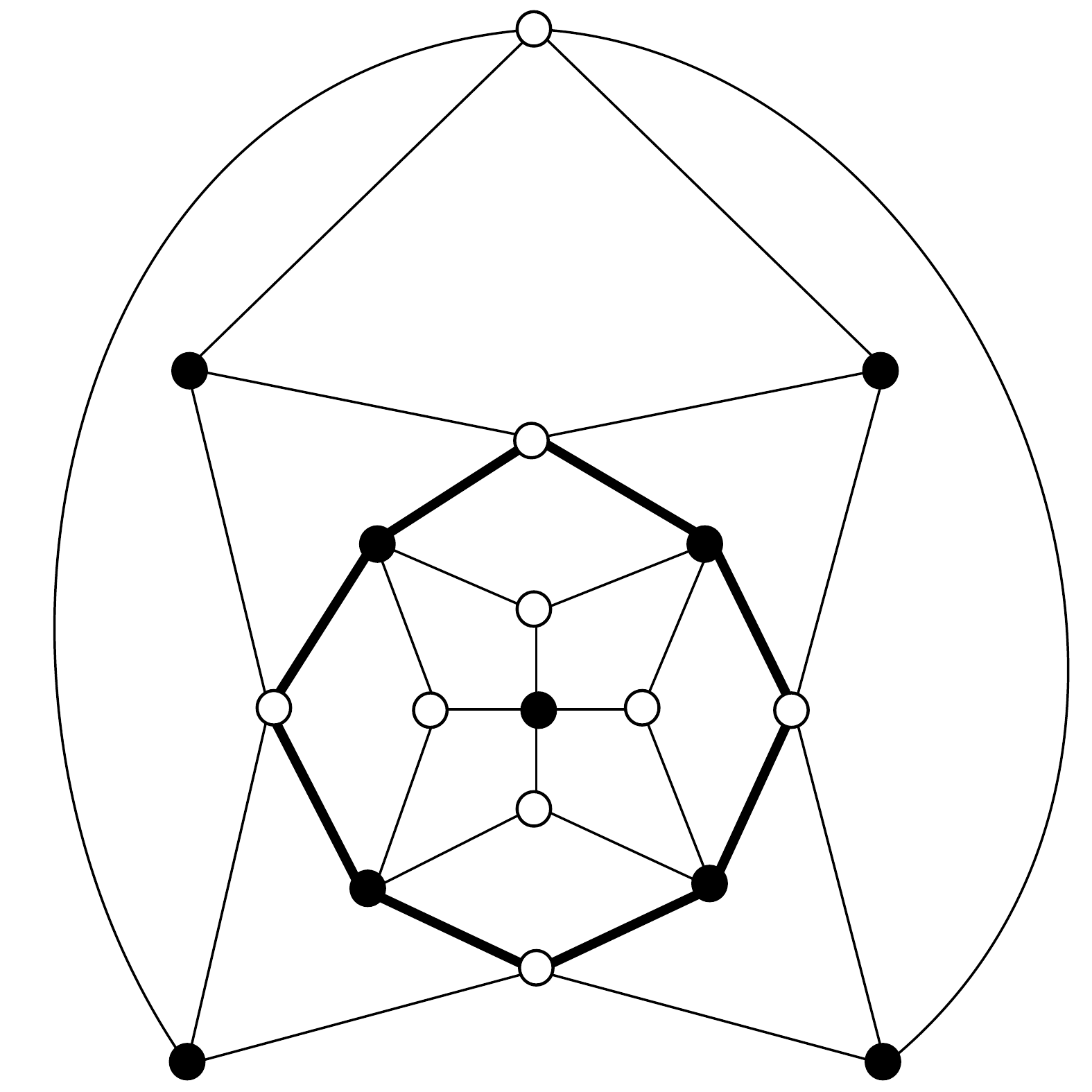}
\caption{$I(P_4^2)$ admitting a symmetric cycle (bold edges) of length 8.}
\label{fig18c1}
\end{figure}

The map of $P_5^2$ given in Figure \ref{fig18d} shows that $(3,2)$-pancake is self-dual antipodal. One can easily mimic this embedding for any odd integer $n\ge 3$ and any $\ell\ge 1$.

\begin{figure}[H]
\centering
\includegraphics[width=0.84\linewidth]{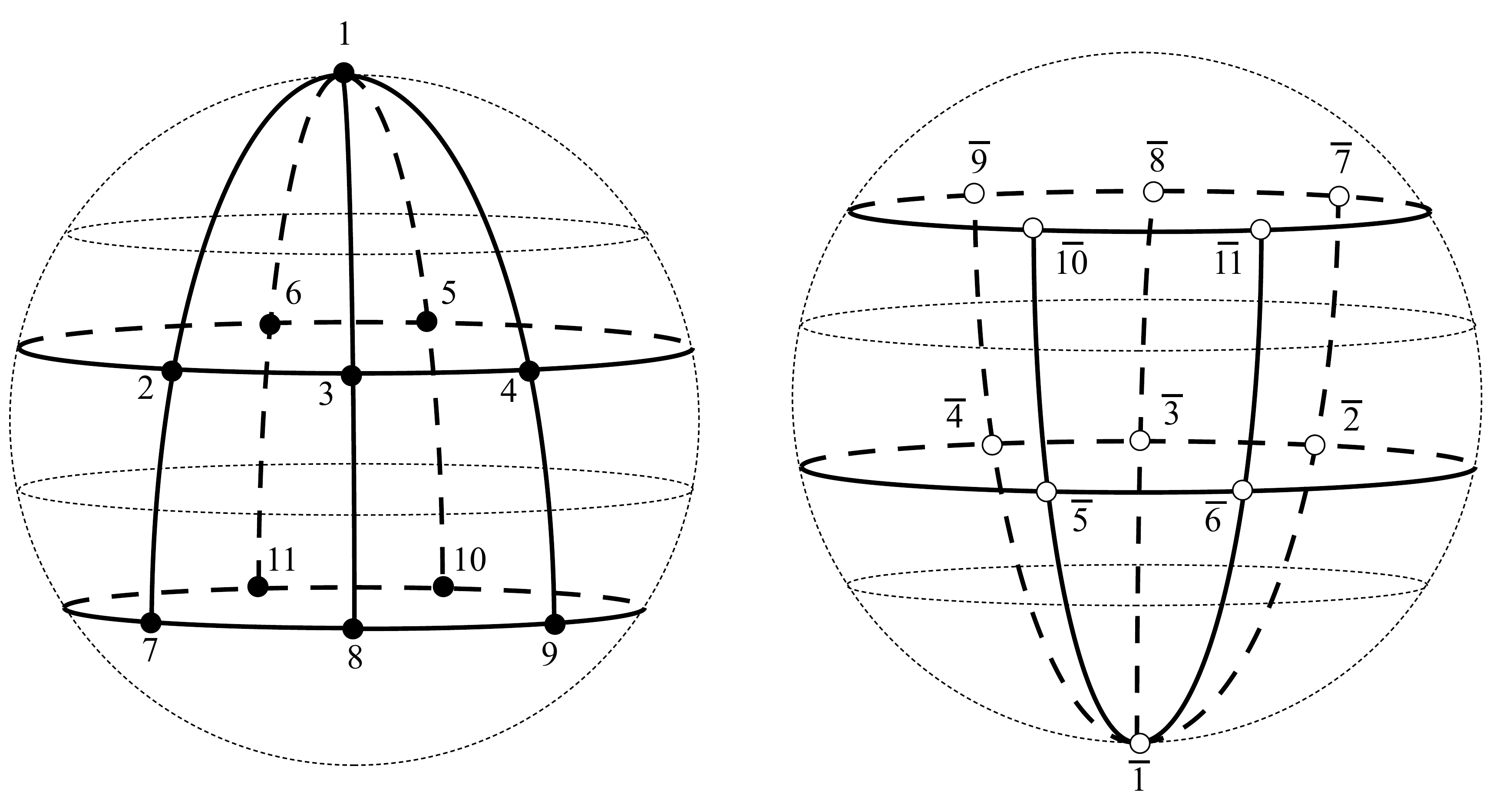}
\caption{A antipodally self-dual map of $P_5^2$ (black vertices) and its dual (white vertices). Antipodal vertices are given by $k$ and $\bar k$.}
\label{fig18d}
\end{figure}

\subsection{Adhesion construction}
Let us give a way to construct infinite families of antipodally self-dual graphs. The latter is based on a procedure to construct self-dual graphs called the {\em adhesion}, given in \cite{SC}. Let $G$ be a planar connected graph and let $G^*$ be its geometric dual. Let $x$ (resp. $x^*$) be the vertex corresponding to the exterior face of $G^*$ (resp. exterior face of $G^{**}=G$). We define the graph $G \diamond G^*$ obtained by identifying $x$ and $x^*$, see Figure \ref{fig7}.
\begin{figure}[H]
\centering
\includegraphics[width=0.6\linewidth]{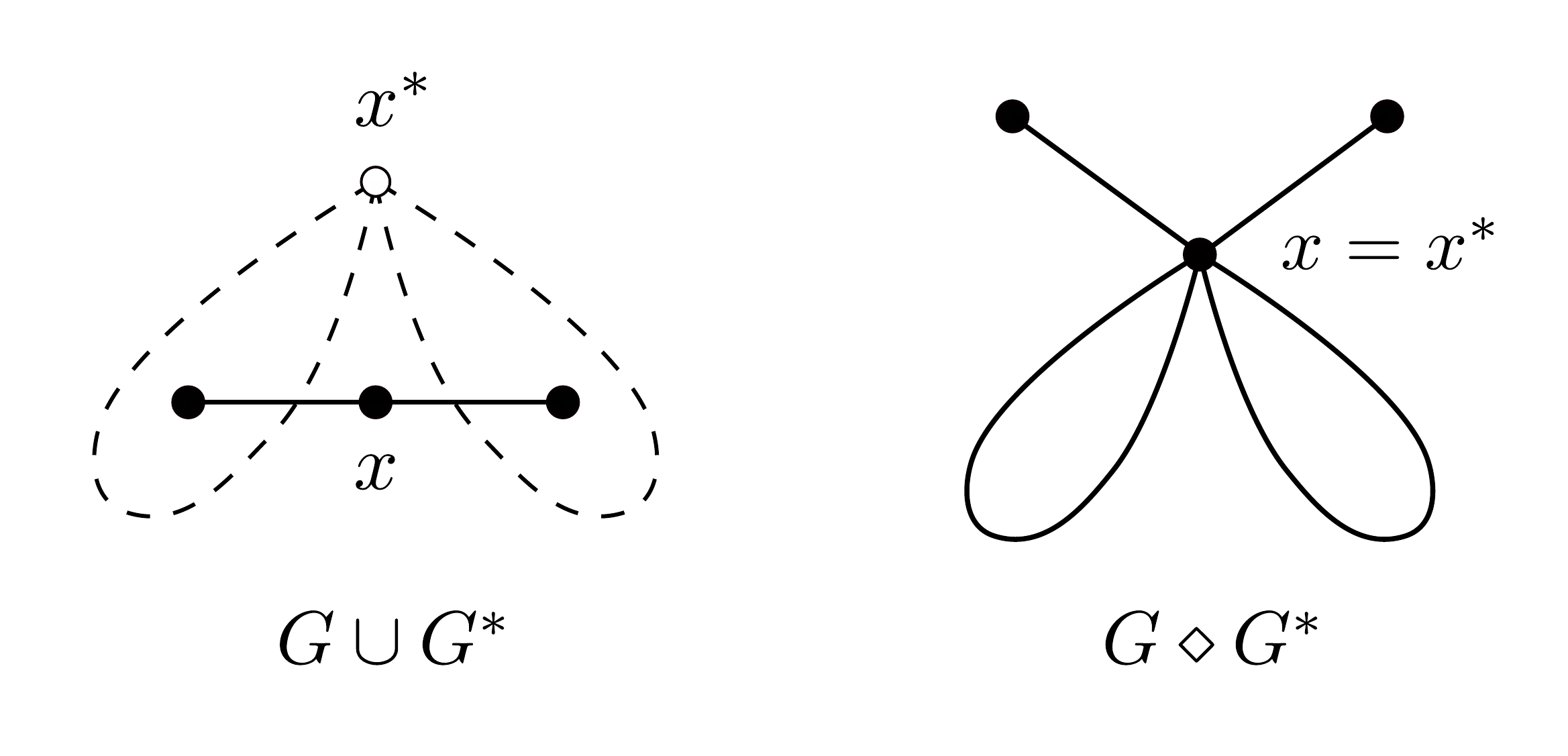}
\caption{(Left) Draw of $G$ and its dual. (Right) The adhesion of $G$.}
\label{fig7}
\end{figure}

\begin{lemma}\label{lem:diamond}\cite{SC} Let $G$ be a planar connected graph. Then, the graph $G \diamond G^*$ is self-dual.
\end{lemma}

\begin{proof}  $H=G \diamond G^*$ is clearly self-dual since $H^*=(G \diamond G^*)^*=G^* \diamond G=G \diamond G^*$.
\end{proof}

Notice that in the construction of $G \diamond G^*$ the couple $x$ and $x^*$ cannot be replaced by \underline{any} pair of vertices since we may end up with a not self-dual graph, see Figure \ref{fig8}.

\begin{figure}[H]
\centering
\includegraphics[width=0.35 \linewidth]{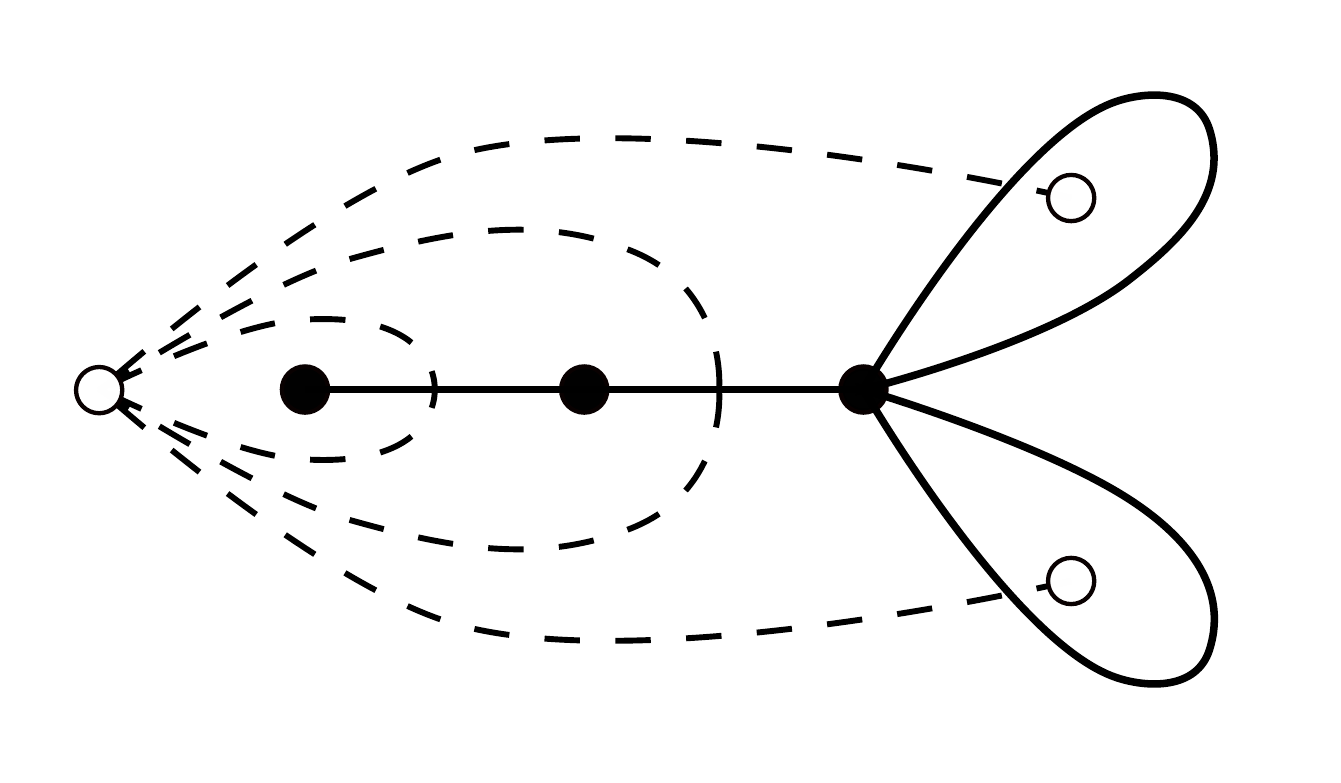}
\caption{A glueing of $G$ and $G^*$ by another pair of vertices which leads to a map which is clearly not self-dual.}
\label{fig8}
\end{figure}
\begin{theorem}\label{thm;adh} Let $G$ be a planar connected graph. Then, $G \diamond G^*$ is antipodally self-dual.
\end{theorem}

\begin{proof} By Lemma \ref{lem:diamond} $H=G \diamond G^*$ is self-dual. Let us show that $H$ admits an antipodal map.  Let $x$ (resp. $x^*$) be the vertex corresponding to the exterior face of $G^*$ (resp. exterior face of $G=G^{**}$). We first draw $G$ and its dual within a circle $C$ such that $x$ and $x^*$ are antipodal points on $C$ and no other edge or vertex (of $G$ or $G^*$) lie on $C$, see Figure \ref{fig9}. 

\begin{figure}[H]
\centering
\includegraphics[width=0.3\linewidth]{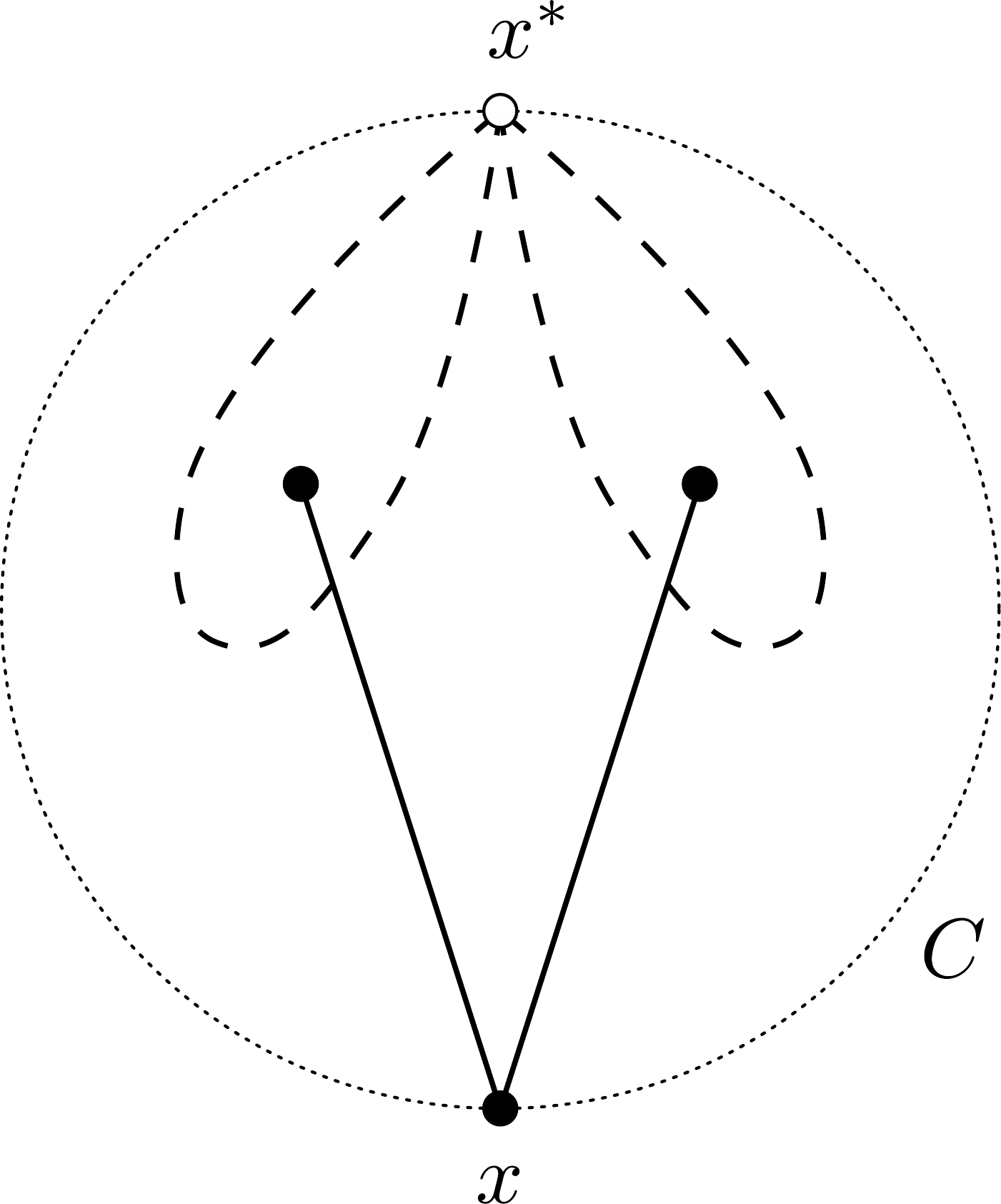}
\caption{Embedding of $G$ and $G^*$ inside circle $C$ with vertex $b$ diametrically opposed to vertex $b^*$.}
\label{fig9}
\end{figure}

We shall construct two embeddings (one in the Northern hemisphere and the other in the Southern one) that will be glued together giving the desired antipodally self-dual embedding of $H$. For, 
we consider $C$ as the equator of $\mathbb{S}^2$ and project our drawing perpendicularly to the Northern hemisphere of $\mathbb{S}^2$. We then take the antipodal of the latter embedding, obtaining and embedding in the Southern hemisphere.
\smallskip

We finally glue together both embeddings along the equator ($x$ and $x^*$ are the only vertices that are identified twice on the equator). By construction, this is an antipodal map of $H$, see Figure \ref{fig10}.  
\end{proof}

\begin{figure}[H]
\centering
\includegraphics[width=.9\linewidth]{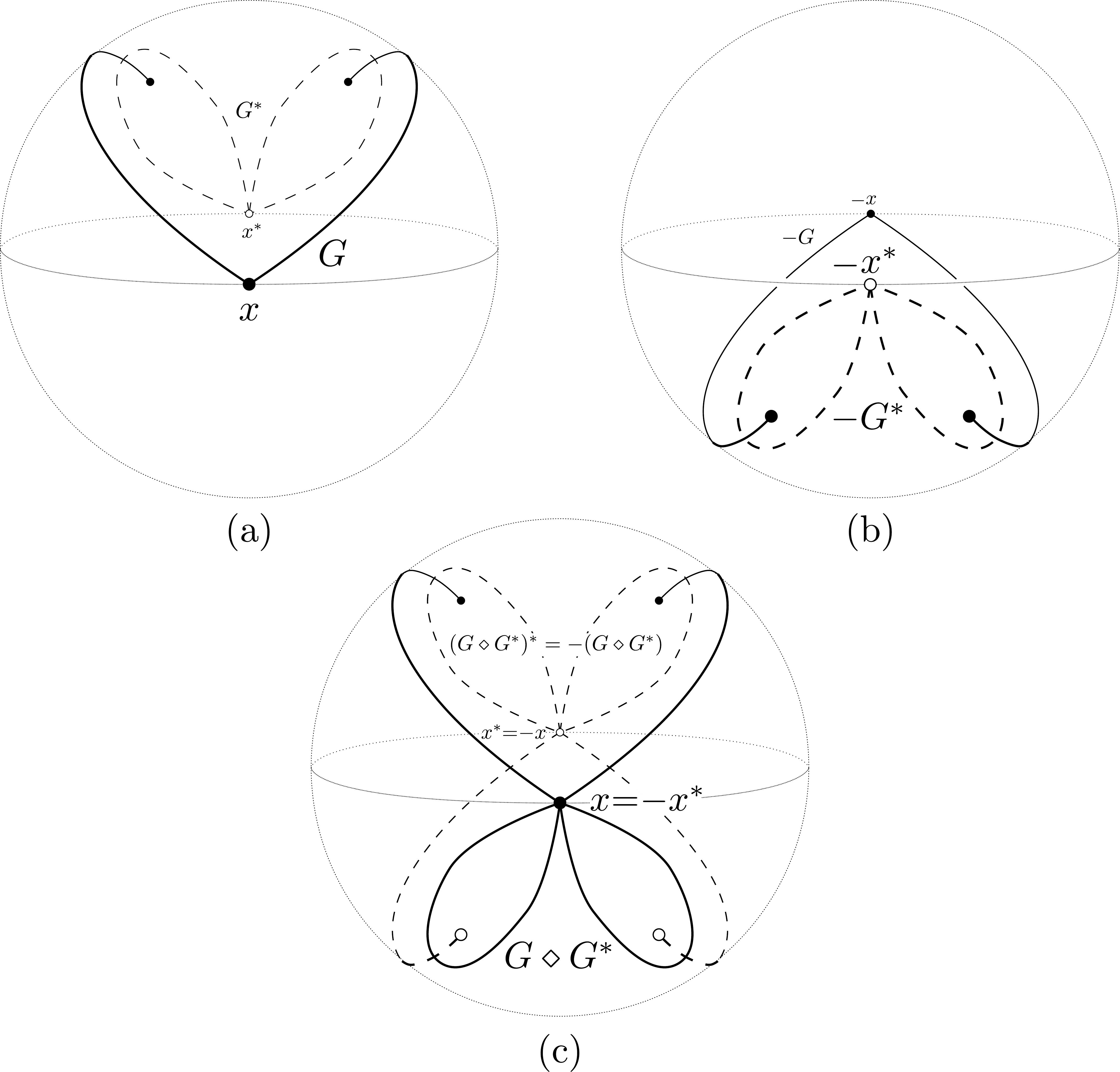}
\caption{(a) Embedding of the draws of $G$ and $G^*$ in the Northern hemisphere (b) Antipodal embedding of the draw of the Northern hemisphere (c) Antipodal embedding of $H=G \diamond G^*$ (bold edges) and $H^*$ (dashed edges).}
\label{fig10}
\end{figure}

\begin{question} Let $H$ be a antipodally self-dual graph with a cut-vertex.  Is it true that  $H=G \diamond G^*$ where $G$ is a planar connected graph and $G^*$ its geometric dual ?
\end{question}


\section{antipodally symmetric maps}\label{sec;self-anti}
A map $G$ is said to be {\em antipodally symmetric} if $-\widehat{G}=\widehat{G}$ where $-G$ is the map consisting of points $\{-x\in \stw \ |\  x\in G\}$. 

\begin{remark} (a) $med(G)=med(G^*)$. 
\smallskip

(b) If $G$ is self-dual then $|V(med(G))|$ is even. Indeed, by Euler's formula we have $|V(G)|+|F(G)|=2+|E(G)|$ where $F(G)$ denote the set of faces of $G$. Since $G$ is self-dual then $|V(G)|=|V(G^*)|=|F(G)|$ and thus $2|V(G)|=2+|E(G)|$ implying that $|E(G)|=|V(med(G))|$ is even.
\end{remark}

\begin{lemma}\label{lem:key} Let $G$ be an antipodally self-dual map. Then, $med(G)$ and $I(G)$ are antipodally symmetric.
\end{lemma}

\begin{proof} We first show that $med(G)$ is self-dual. For, let us consider a antipodally self-dual map $G$, that is, the dual map $G^*$ is antipodally embedded with respect to the map $G$. The latter induces a map $G^\Box$ in which square faces of $G^\Box$ are partitioned into pairs that are antipodally embedded in $\mathbb{S}^2$. Indeed, let $F=\{e_1,e_2,e_1^*,e_2^*\}$ be a face of $G^{\Box}$ where $e_1,e_2$  (resp. $e_1^*,e_2^*$) are the two half-edge induced by $e\in E(G)$ (resp. induced by $e^*\in E(G^*)$). Since $G$ is antipodally self-dual then there is an edge $f^*\in G^*$ (resp. an edge $f\in G$) which is antipodally embedded to $e\in G$ (resp. to $e^*\in G^*$). We thus have that the corresponding half-edges $f_1^*,f_2^*$ (resp. $f_1,f_2$) are also antipodally embedded with respect to $e_1,e_2$ (resp. to $e_1^*,e_2^*$). Obtaining an other face $F^*=\{f_1,f_2,f_1^*,f_2^*\}$ which is antipodally embedded with respect to $F$.
\smallskip

We thus have that the intersecting diagonals corresponding to faces $F$ and $F^*$ can also be antipodally embedded. The results follows by recalling that  $med(G)$ is given by all the intersecting diagonals of $G^\Box$, see Figure \ref{fig13}.

\begin{figure}[H]
\centering
\includegraphics[width=0.8\linewidth]{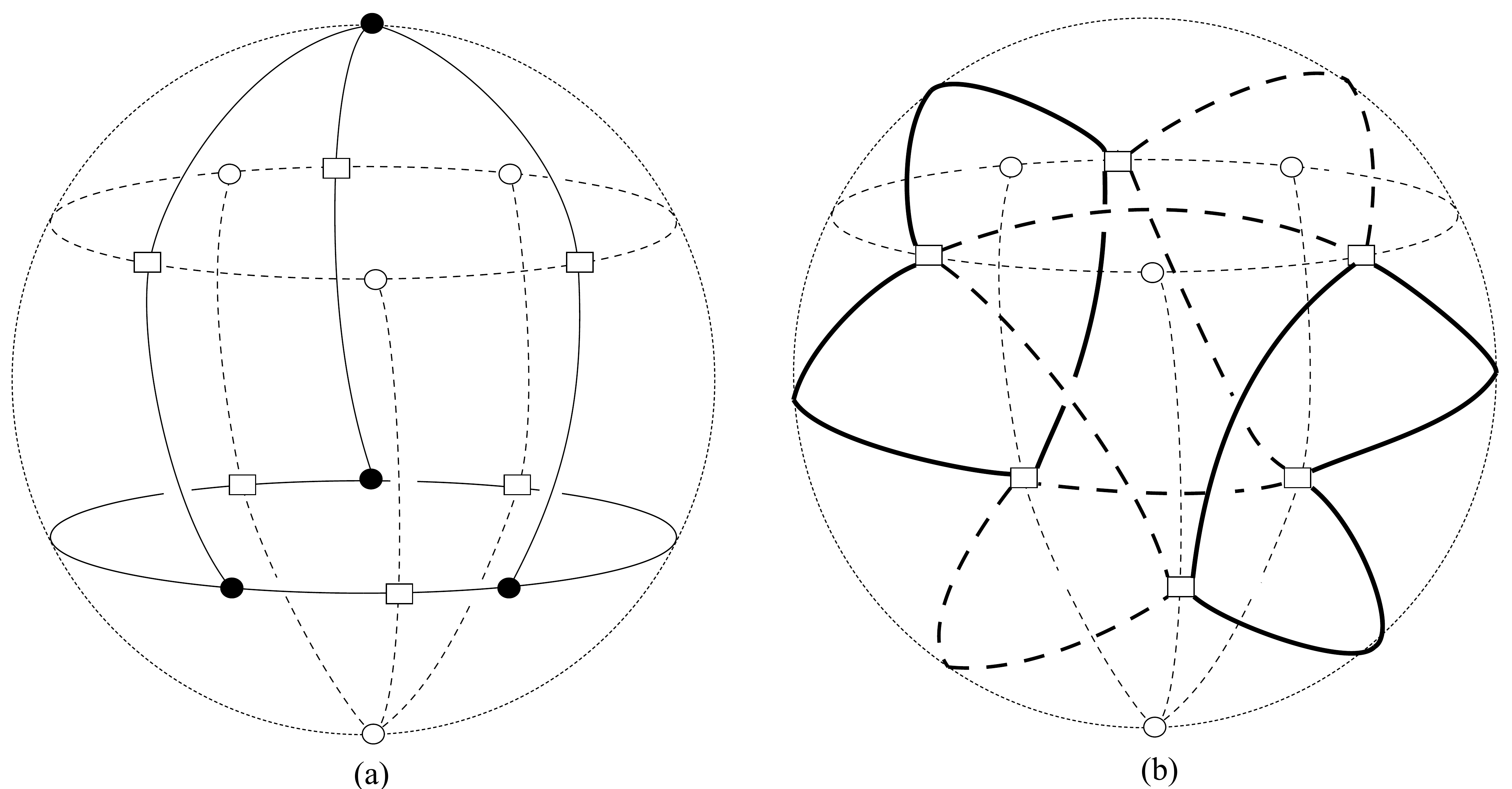}
\caption{(a) Embedding of $K_4$ (black vertices and straight edges), its dual (white vertices and dashed edges) and the vertices of the medial graph (little squares vertices) (b) $K_4$ (dashed edges) and its medial graph (in bold dashed edges) with two antipodal faces (in bold).}
\label{fig13}
\end{figure}

For $I(G)$, the proof goes in the same way as above but, this time, by considering the incidence diagonals instead of the intersecting diagonals.
\end{proof}

We end this section by proving Theorem \ref{theom;ant1}.
\smallskip

{\em Proof of Theorem \ref{theom;ant1}.}  Let $G$ be a antipodally self-dual map. Let
$\widehat{med(G)}$ be the drawing of $med(G)$ where all the automorphisms are isometries and let $E$ be the equator of $\stw$.
\smallskip

Suppose that $E$ does not contain any vertex of $\widehat{med(G)}$; Then, $E$ passes from a face $f$ of $\widehat{med(G)}$ to another face $f'$ that shares and edge with $f$. Since $med(G)^*=I(G)$ the pair faces $\{f,f'\}$ corresponds to a pair of adjacent vertices $\{v,v'\}$ in $V(I(G))$. Thus, the sequence of faces $(f_1,\ldots,f_n=f_1)$ intersected by $E$ (with the order induced by $E$) corresponds to a cycle $C$ in $I(G)$. Let $int(C)$ (resp. $ext(G)$) be the subgraph of $I(G)$ corresponding to the faces of $\widehat{med(G)}$ lying on the northern (resp. southern) hemisphere. By Lemma \ref{lem:key}, $med(G)$ is antipodally symmetric so the northern faces and the southern faces of $\widehat{med(G)}$ are antipodally drawn. Thus, $int(C)$ is map isomorphic to $ext(C)$ and thus $C$ is a symmetric cycle of $I(G)$.
\smallskip

Now, let us suppose that $E$ passes through a vertex of $med(G)$. Since the set of vertices of $med(G)$ is finite there exists a point $x\in E$ such that $x$ and $-x$ are not vertices of $med(G)$. Let $E_\alpha$ be the the rotation of $E$ of angle $\alpha$ on the line passing through $x$ and $-x$. Let $$\beta=\min_{\alpha>0}\{E_\alpha\text{ contains a vertex of }\widehat{med(G)}\}.$$ 
Then, $E_{\beta/2}$ is a great circle of $\stw$ which does not contain any vertex of $\widehat{med(G)}$. By taking $E_{\beta/2}$ as equator we can apply the above arguments to show that there exists a symmetric cycle of $I(G)$. 
\smallskip

Finally, if $C$ is a symmetric cycle of $I(G)$ then we can draw $\widehat{I(G)}$ with $C$ being the equator of $\stw$. Since $G$ is antipodally self-dual, a black vertex $v$ of $\widehat{I(G)}$ is antipodal to a white vertex $-v$. Thus, the length of $C$ must be $2n$ with $n\ge 1$ odd.  \hfill$\square$



\begin{thebibliography}{99}

\bibitem{BMPRA} J. Bracho, L. Montejano, E. Pauli and J.L. Ram\'irez Alfons\'in, Strongly involutive self-dual polyhedra, arXiv:2005.03866


 \bibitem{E} P. Erd\"os, On sets of distances of $n$ points, {\em Amer. Math. Monthly} {\bf 53} (1946) 248-250.
 
\bibitem{GS} B. Gr\"unbaum and G.C. Shepard, Is selfduality involutory ?, {\em Amer. Math. Monthly} {\bf 95} (1985), 729-733.

\bibitem{L} L. Lov\'asz, Self-dual polytopes and the chromatic number of distance graphs on the sphere, {\em Acta Sci. Math.} {\bf 45} (1983), 317-323.

\bibitem{MMO} H. Martini, L. Montejano and D. Oliveros, Bodies of Constant width;
An introduction to convex geometry with applications, {\em Birkh\"user} (2019).

\bibitem{MRAR2} L. Montejano, J. L. Ram\'irez Alfons\'in and I. Rasskin, Self-dual maps II: links and symmetry

\bibitem{SC} B. Servatius and P.R. Christopher, Construction of self-dual graphs, {\em Amer. Math. Monthly}  {\bf 99}(2) (1992), 153-158.

\bibitem{SS} B. Servatius and H. Servatius, The 24 symmetry pairs of self-dual maps on the sphere, {\em Disc. Math.} {\bf 140} (1995), 167-183.

\bibitem{SS1} B. Servatius and H. Servatius, Symmetry, automorphisms and self-duality of infinite planar graphs and tilings, In {\em International  Scientific  Conference on Mathematics.  Proceedings  
(\u{Z}ilina, 1998)},  pages  83–116. Univ. \u{Z}ilina, \u{Z}ilina, 1998.

\bibitem{KS} K.J. Swanepoel, A new proof of V\'azsonyi's conjecture, {\em J. Comb. Th. Ser. A} {\bf 115} (2008), 888-892.

\bibitem{T} R.M. Tifenbach, Strongly self-dual graphs, {\em Lin. Alg. and its Appl.} {\bf 435} (2001), 3151-3167.


\end{thebibliography}
\end{document}